\documentclass[gdweb,gdminitabs,gdfinal]{geradwp}

\PassOptionsToPackage{hyphens}{url}

\usepackage[ruled,linesnumbered]{algorithm2e}
\usepackage{hyperref}

\GDcoverpagewhitespace{6.8cm}
\graphicspath{{Figures/}} 
\hypersetup{colorlinks,%
citecolor={blue}, 
urlcolor={blue},
linkcolor={blue},
breaklinks={true}
}

\makeatletter
\ifthenelse{\isundefined{\ALG@name}}{}%
{%
\renewcommand{\ALG@name}{\sffamily\footnotesize Algorithm}
}
\makeatother
\ifthenelse{\isundefined{\SetAlCapNameFnt}}{}%
{%
\SetAlCapNameFnt{\footnotesize}
\SetAlCapFnt{\sffamily\footnotesize}
}

\usepackage[sort,numbers]{natbib}
\usepackage{amsthm}
\theoremstyle{plain}

\newtheorem{theorem}{\sffamily Theorem}

\theoremstyle{definition}
\newtheorem{definition}{\sffamily Definition}

\theoremstyle{remark}
\newtheorem{remark}{\sffamily Remark}

\newcolumntype{C}[1]{>{\centering\arraybackslash}p{#1}}

\usepackage{tikz}
\usepackage{float}
\usepackage{rotating}
\usepackage{adjustbox}
\usepackage{makecell}
\usepackage{multicol}
\usetikzlibrary{positioning,chains,fit,shapes,calc,arrows.meta,                        
	  backgrounds,                        
	  ext.paths.ortho,                    
	  ext.positioning-plus,               
	  }
\usepackage{pgfplots}
\pgfplotsset{compat=1.18}
\usetikzlibrary{spy}
\usepackage{newtxtext,newtxmath}
\usepackage{svg}
\usepackage{fontawesome}
\usepackage{amsthm}
\usepackage{mathtools}
\usepackage{enumitem}
\usepackage{soul}
\soulregister\cite7
\soulregister\ref7
\soulregister\pageref7
\usepackage[colorinlistoftodos]{todonotes}

\GDtitle{Towards a connection between the capacitated vehicle routing problem and the constrained centroid-based clustering}
\GDmonth{March}{Mars}
\GDyear{2024}
\GDnumber{24}
\GDauthorsShort{A.~Abdellaoui , L.~Benabbou, I.~El~Hallaoui}
\GDauthorsCopyright{Abdellaoui, El~Hallaoui, Benabbou}
\GDpostpubcitation{Hamel, Benoit, Karine H\'ebert (2024). ``Un exemple de citation'', \emph{Journal of Journals}, vol. X issue Y, p. n-m}{https://www.gerad.ca/fr}
\GDsupplementname{Internet Appendix}
\GDrevised{Mai}{May}{2024}

\begin{document} 

\GDcoverpage

\begin{GDtitlepage}

\begin{GDauthlist}
	\GDauthitem{Abdelhakim Abdellaoui \ref{affil:bib}\GDrefsep\ref{affil:gerad}}
	\GDauthitem{Loubna Benabbou \ref{affil:gerad}\GDrefsep\ref{affil:hec}}
	\GDauthitem{Issmail El Hallaoui \ref{affil:gerad}\GDrefsep\ref{affil:bib}}
\end{GDauthlist}

\begin{GDaffillist}
	\GDaffilitem{affil:bib}{Polytechnique Montréal,, Département de Mathématiques et de Génie Industriel, Montr\'eal (Qc), Canada, QC H3T 1J4}
	\GDaffilitem{affil:gerad}{GERAD, Montr\'eal (Qc), Canada, H3T 1J4}
	\GDaffilitem{affil:hec}{UQAR L\'evis campus, Université du Québec à Rimouski, G6V 0A6}
\end{GDaffillist}

\begin{GDemaillist}
	\GDemailitem{abdelhakim.abdellaoui@polymtl.ca}
	\GDemailitem{loubna\_benabbou@uqar.ca}
	\GDemailitem{issmail.elhallaoui@polymtl.ca}
\end{GDemaillist}

\end{GDtitlepage}


\GDabstracts

\begin{GDabstract}{Abstract}
Efficiently solving a vehicle routing problem ($\mathcal{VRP}$) in a practical runtime is a critical challenge for delivery management companies. This paper explores both a theoretical and experimental connection between the Capacitated Vehicle Routing Problem  ($\mathcal{CVRP}$) and the Constrained Centroid-Based Clustering ($\mathcal{CCBC}$). Reducing a $\mathcal{CVRP}$ to a $\mathcal{CCBC}$ is a synonym for a transition from an exponential to a polynomial complexity using commonly known algorithms for clustering, i.e \textit{K-means}. At the beginning, we conduct an exploratory analysis to highlight the existence of such a relationship between the two problems through illustrative small-size examples and simultaneously deduce some mathematically-related formulations and properties. On a second level,  the paper proposes a $\mathcal{CCBC}$ based approach endowed with some enhancements. The proposed framework consists of three stages. At the first step, a constrained centroid-based clustering algorithm generates feasible clusters of customers. This methodology incorporates three enhancement tools to achieve near-optimal clusters, namely: a multi-start procedure for initial centroids, a customer assignment metric, and a self-adjustment mechanism for choosing the number of clusters. At the second step, a traveling salesman problem ($\mathcal{TSP}$) solver is used to optimize the order of customers within each cluster. Finally, we introduce a process relying on routes cutting and relinking procedure, which calls upon solving a linear and integer programming model to further improve the obtained routes. This step is inspired by the \textit{ruin \& recreate} algorithm. This approach is an extension of the classical \textit{cluster-first, route-second} method and provides near-optimal solutions on well-known benchmark instances in terms of solution quality and computational runtime, offering a milestone in solving $\mathcal{VRP}$.

\end{GDabstract}

%
%
%
%


\GDarticlestart
\section{Introduction}
The vehicle routing problem ($\mathcal{VRP}$) is defined as the exercise of finding the best vehicle routes to deliver products to a set of customers \cite{bertazzi2015min}. In practice, it often comes with different constraints reflecting the business nature. Recognizing its critical importance in various domains, it has been extensively studied by the operations research community over the past few decades. In particular, several approaches were designed to solve this problem, known as $\mathcal{NP}$-hard. Hence, it is always worthwhile to design new methodologies to achieve more efficient solutions within a practical timeframe. From this standpoint, the machine learning community has recently been  more involved in tackling this problem. Consequently, many techniques have been tested while building $\mathcal{VRP}$  solvers: clustering \cite{crainic2008clustering}, reinforcement learning \cite{nazari2018reinforcement}, and learning over graphs \cite{khalil2017learning}.

In this paper, we focus on leveraging the clustering techniques to highlight the connection that could be established between the capacitated vehicle routing problem ($\mathcal{CVRP}$) and the constrained centroid-based clustering ($\mathcal{CCBC}$). Establishing this connection can be highly advantageous for the operations research community, because there exists a valuable body of knowledge within the clustering community that can be harnessed to enrich $\mathcal{VRP}$ solution methodologies. The objective of this paper is to design a $\mathcal{CVRP}$ solver using a $\mathcal{CCBC}$ technique to reach good quality solution within reasonable runtime. This choice can be rationalized by the fact that the clustering reduces the huge combinatorial space through dealing with sub-instances instead of tackling the raw problem. This can be seen as a direct application of the divide-and-conquer paradigm to solve the $\mathcal{CVRP}$. Nevertheless, a straightforward implementation of this approach has some limitations and does not always guarantee good quality solutions. It inherits the same shortcomings from the classical clustering algorithms, specifically \cite{he2009balanced,ewbank2016unsupervised}: choice of initial centroids, local optima, unbalanced clusters, and border points.

In this work, we address four contributions. Particularly, we target defining the nature of the connection that can be set up between $\mathcal{CVRP}$ and $\mathcal{CCBC}$ and the related limitations. Additionally, the focus will be centered on the strategy to mitigate the impact of these limitations. This will result in providing $\mathcal{CVRP}$  solutions in a reasonable runtime while ensuring good quality. To the best of our knowledge, it is
the first time such an approach has been applied to tackle the $\mathcal{CVRP}$.
To sum up, we list below the main contributions of this paper: 
\begin{enumerate}
	\item Highlight through experiments the connection between the $\mathcal{CVRP}$ and $\mathcal{CCBC}$ and prove some theoretically related properties.
	\item Design a $\mathcal{CCBC}$ based approach to address the $\mathcal{CVRP}$ and relieve the shortcomings impact mentioned above.
	\item Provide a computational study on baseline instances from the literature showing near-optimal solution, resulting in an average gap of 1.07 \% to the optimal solution.
	\item Carry out an analysis study to shed light on the impact of added enhancements against the aforementioned shortcomings.	
\end{enumerate}

The remainder of this paper is organized as follows. Section \ref{lit} is dedicated to the literature review. Section~\ref{prbfor} is devoted to the problem formulation while Section~\ref{connect-cvrp_ccbc} introduces a warm-up study to highlight the connection between $\mathcal{CVRP}$ and $\mathcal{CCBC}$. Section~\ref{solmeth} presents the methodology while the experimentation and the post-computational analysis is given in Section~\ref{exper}. The last section provides conclusion and perspectives.

\section{Literature review} \label{lit}
This section presents a literature review about $\mathcal{VRP}$ characteristics and the common methodologies to solve it. The focus will especially be given to the use of the clustering-driven approaches to tackle this problem.

\subsection{\texorpdfstring{$\mathcal{VRP}$}{} variants}

The Vehicle Routing Problem ($\mathcal{VRP}$) and its numerous variants, estimated to be over 10 in number~\cite{elshaer2020taxonomic}, are 
universally known to be $\mathcal{NP}$-hard problems~\cite{lenstra1981complexity}. The capacitated vehicle routing problem ($\mathcal{CVRP}$)~\cite{dantzig1959truck} considered as the foundational variant and consists in serving all clients through a set of vehicles. Each vehicle starts from and ends at the depot, such that every route's total demand must not exceed the vehicle capacity. We outline a few notable variants such as :
\begin{description}
\item \textbf{Vehicle routing problem with time windows $\mathcal{VRPTW}$  \cite{solomon1987algorithms}:} This is an extension of $\mathcal{CVRP}$, requiring a soft or hard time windows to serve each client. 

\item \textbf{Vehicle routing problem with pickup and delivery $\mathcal{VRPPD}$ \cite{desaulniers2002vrp}:} This variant adds the complexity of handling both deliveries and pickups at customer locations.

\item \textbf{Dynamic Vehicle Routing Problem $\mathcal{DVRP}$ \cite{pillac2013review}:} This version adapts to real-time changes, such as customer requests or traffic conditions, during the operation.
\end{description}

Our paper will primarily focus on the $\mathcal{CVRP}$ variant. This choice is due to the fact that the complexity of the other variants largely originates from the intrinsic complexity of $\mathcal{CVRP}$, which serves as the cornerstone for all variants \cite{archetti2011complexity}. In the following sections, we present an overview of the classical approaches for solving $\mathcal{VRP}$ variants : Metaheuristics, heuristics, and exact methods. Then, the main concern will be centered on using the clustering techniques to solve $\mathcal{VRP}$ variants.

\subsection{Operations research techniques for solving \texorpdfstring{$\mathcal{VRP}$}{}}
Within the operations research community, one can categorize the commonly-used techniques for solving $\mathcal{VRP}$ into many classes, namely: exact approaches, heuristics, meta-heuristics, etc..
Exact methods tend to find optimal solutions but can be computationally expensive for large-scale $\mathcal{VRP}$. The work \cite{laporte1992vehicle} introduced a taxonomic overview of the foundations of exact methods for solving $\mathcal{VRP}$ variants. Furthermore, the paper \cite{baldacci2012recent} reviewed recent advancements within the exact solution approaches, focusing on mathematical formulations and relaxations used to address popular $\mathcal{VRP}$ variants, including capacity and time windows constraints. Using an exact approach, \cite{pecin2017improved} was capable of solving the majority of benchmark instances from the literature up to a size of 275 customers using a \textit{Branch-cut-and-price method}.

Many heuristics were presented to approximate near-optimal solutions for $\mathcal{VRP}$ variants. One can cite route-building heuristics, which are iteratively performed to combine customers in a route relying on specific criteria. The algorithm designed in \cite{baker1986solution} used branch exchange procedures (\textit{2-opt, 3-opt }) to gradually design routes with the maximum saving starting from singleton sets. 
In this context, the survey \cite{braysy2005vehicle1} shed light on route construction heuristics and local search algorithm to solve $\mathcal{VRPWT}$.

The family of meta-heuristics has been tremendously successful in practice for solving difficult combinatorial optimization problems. Unlike heuristics which are problem-dependent, meta-heuristics do not require any prior knowledge. Therefore, they are applied to a broad range of problems. One can mention the well-known ones such as simulated annealing, genetic algorithm, tabu search ...In this regard, the paper \cite{elshaer2020taxonomic} presented a taxonomic review of the existing meta-heuristics approaches in the literature to solve $\mathcal{VRP}$ variants. Additionally, \cite{braysy2005vehicle2} introduced a survey on the meta-heuristics for $\mathcal{VRPTW}$. 

\subsection{Clustering-based approaches for solving \texorpdfstring{$\mathcal{VRP}$}{}}
Given on the one hand that exact methods often fail to solve instances with more than 300 customers due to the problem complexity \cite{uchoa2017new}, and on the other hand, heuristics are less adaptable when it comes to changes within the problem, e.g., customer demand or position \cite{nazari2018reinforcement}, several studies have consequently shifted their focus towards solving the $\mathcal{VRP}$ using other techniques, e.g., by leveraging machine learning techniques. Clustering techniques have first been used for analysis purposes within the supply chain sector for many tasks, such as item partitioning in inventory \cite{axsater1981aggregation}, production \cite{ernst1990operations}, E-business \cite{cagliano2003business}. When it comes to $\mathcal{VRP}$, the clustering approach was first called upon to assist heuristics and meta-heuristics while solving VRP. In this context, many works made use of straightforward clustering methods such as : local search, random geographical partition of the area~\cite{barthelemy2010metaheuristic,gillett1974heuristic}. These techniques were applied in conjunction with heuristics and meta-heuristics such as simulated annealing, the sweep algorithm to reduce $\mathcal{VRP}$ complexity before designing the final solution \cite{kirkpatrick1983optimization,gillett1974heuristic,hiquebran1993revised}. 
Over the past few decades, the operations research community has leveraged the growth of machine learning to design clustering-based approaches to tackle the $\mathcal{VRP}$ using algorithms, namely: \textit{Kmeans, Fuzzy Cmean, Dbscan, and neural networks}.  Depending on the specific objectives of using clustering when solving  $\mathcal{VRP}$, one can classify the existing clustering-based approaches for solving $\mathcal{VRP}$ into three categories, more details and references are introduced in Table \ref{summVRP}:
\begin{description}
\item[Cluster-first, route-second \cite{kirkpatrick1983optimization} :  ] The initial step of this framework involves the clustering of customers, where each resulting cluster represents a feasible and unordered route. Subsequently, it proceeds to address each cluster independently by solving a $\mathcal{TSP}$. In the same context, the papers \cite{ewbank2019capacitated,ewbank2016unsupervised} made use of the well known \textit{fuzzy c-means} combined with a learning approach to define the relevant fuzziness parameter. Once this parameter is determined, the customers are assigned to clusters. Therefore, each cluster is handed to a $\mathcal{TSP}$ solver to design the final routes. 

\item[Clustered $\mathcal{VRP}$ \cite{barthelemy2010metaheuristic} ( \textit{CluVRP}): ] This is a transformation of the original $\mathcal{VRP}$ into a compact variant in which customers are partitioned into small clusters. A distinctive characteristic of the \textit{CluVRP} is that when a vehicle visits a customer, it must subsequently visit all the remaining customers in the cluster. After that, it approaches the problem using exact or non-exact methods to generate routes. As illustration, \cite{dondo2007cluster} capitalized on a customized clustering heuristic to aggregate clients into macro-nodes as an initial step to reduce the problem size. Then, an MILP is used to design routes through sequencing macro-nodes with respect to known constraints and considering a heterogeneous fleet of vehicles and multiple depots. Similarly, \cite{yucenur2011new} proposed a geometric shape-based genetic clustering algorithm to deal with a multi-depot vehicle routing problem. More works are presented in Table \ref{summVRP}. 

\item[Large-scale $\mathcal{VRP}$ decomposition : ] This technique is mainly designed for addressing large-scale $\mathcal{VRP}$. It stands as an effective approach, gaining recognition for its real-world applicability. The core idea involves clustering customers into groups to reduce the problem's complexity. Notably, this approach has shown its efficiency in enhancing the performance in terms of the runtime \cite{kirkpatrick1983optimization,gillett1974heuristic,hiquebran1993revised}. Numerous papers have recently used the known existing clustering algorithms to decompose the $\mathcal{VRP}$. The paper \cite{he2009balanced} designed a modified \textit{Kmeans} version with a border adjustment feature to get balanced clusters. In the same context, the paper \cite{bujel2018solving} relied on a recursive \textit{Dbscan} to partition large-scale $\mathcal{VRP}$. 
\end{description}

\begin{table}[h!]
	\begin{tabular}{lll} 
	\toprule
	\textbf{Class}& \textbf{Authors}& \textbf{Clustering algorithm} \\
	\midrule
	\multirow{2}{4cm}{Cluster first, route second ($\mathcal{CFRS}$)} & Ewbank {et al.}\cite{ewbank2019capacitated}  & Fuzzy c-means + Neural network \\
	&Ewbank~{et al.} \cite{ewbank2016unsupervised} & Fuzzy c-means + Neural network\\ 
	\midrule
	\multirow{4}{4cm}{Clustered $\mathcal{VRP}$( \textit{CluVRP} )} & 
	Bektas~{et al.}\cite{bektacs2011formulations}, Hintsch~{et al.} \cite{hintsch2020exact}  & Clustering \\ 
	& Barthelemy~{et al.} \cite{barthelemy2010metaheuristic}, Vidal~{et al.} \cite{vidal2015hybrid} & Iterated local search\\
	&Dondo~{et al.} \cite{dondo2007cluster} & Customized clustering procedure\\ 
	& Alesiani~{et al.} \cite{alesiani2022constrained} & Self-adapted K-means\\ 
	\midrule
	\multirow{3}{4cm}{Large-scale $\mathcal{VRP}$ decomposition} & He~{et al.} \cite{he2009balanced} & Balanced K-means \\
	& Bujel~{et al.} \cite{bujel2018solving} & Recursive Dbscan\\ 
	& Gillett~{et al.} \cite{gillett1974heuristic} & Random geographical partition\\ 
	\bottomrule
	\caption{Summary of clustering-based approaches for solving $\mathcal{VRP}$ variants.}
	\label{summVRP}
	\end{tabular}
\end{table}	

Through this literature review, we can notice that the existing clustering-based approaches put significant emphasis on the post-clustering phase when solving the $\mathcal{VRP}$. Concretely, the main contributions of the majority of these papers are not predominantly related to the clustering aspect itself but rather revolve around the development of exact or non-exact methods for generating final routes from the designed clusters. Additionally, the cited papers rely on the clustering as preprocessing scheme to reduce the $\mathcal{VRP}$ complexity, thereby improving the runtime performance. In contrast to that, the current paper focuses on improving the clustering step in order to obtain better performance in terms of both runtime and quality solution. This is accomplished through an in-depth analysis of the connection we can establish between the clustering and the $\mathcal{CVRP}$ as elaborated in Section \ref{connect-cvrp_ccbc}.

\section{Problem statement} \label{prbfor}
Given our intention in this paper to establish  theoretical and experimental connections, between $\mathcal{CVRP}$ and $\mathcal{CCBC}$, we divide this section into three parts. The first part introduces the used notation throughout the paper. Then, the second one sheds light on $\mathcal{CVRP}$ related concepts and the third one presents the centroid-based clustering approach.

\subsection{Mathematical notation}
We first provide a detailed overview of the mathematical notation used throughout this paper. The sets, indices, parameters and decision variables
are introduced in Table \ref{vrpformulation}.

\subsection{Capacitated vehicle routing problem}
The $\mathcal{VRP}$ can be defined with respect to many constraints. We limit our current study to the capacity constraint. From a practical perspective, we can describe the $\mathcal{CVRP}$
problem as follows: a set of customers are located in different places and should be served through a fleet of vehicles. Each vehicle has a well-known capacity and has the same starting and ending point called depot as illustrated in Figure \ref{cvrp}. The goal is to determine a dispatching plan that minimizes the total traveled distance. Each customer must be visited only once, and the total demand of a complete route must not exceed the vehicle capacity. Furthermore, it should be noted that the current study is limited to a homogeneous fleet. This problem is formulated using a mixed integer linear program \eqref{objfun}--\eqref{eq71}. 
\begin{align}
	\min & \,\,\,\, \sum_{v=1}^{M} \sum_{i=0}^{N}\sum_{j=0, \,i \neq j}^{N}d_{ij}x_{ij}^{v}
	& & \label{objfun}\\
	\text{s.t.} &\,\,\,\, \sum_{v=1}^{M} \sum_{i=0,i\neq j}^{N}x_{ij}^{v}=1 
	&& \forall j \in \mathcal{C}  \label{eq211}\\
	&\,\,\,\, \sum_{i=1}^{N}x_{0i}^{v} \leq 1 &&  \forall v \in \mathcal{V} \label{eq311}\\
	&\,\,\,\,\sum_{i=0,\,i\neq j}^{N}x_{ij}^{v}-\sum_{i=0,\,i\neq j}^{N}x_{ji}^{v}=0  && \forall v \in \mathcal{V}, \forall j \in \mathcal{C}\label{eq41}\\
	&\,\,\,\,\sum_{j=0}^{N} \sum_{i=0,\, i\neq j}^{N}q_j x_{ij}^{v} \leq Q   && \forall v \in \mathcal{V}\label{eq51}\\
	&\,\,\,\, \sum_{v=1}^{M} \sum_{i  \in S}\sum_{j  \in  S, i \neq j} x_{ij}^v\leq |S|-1  &&\forall  S \subseteq \mathcal{C},  2 \leq |S| \leq  N\label{eq61}\\
	&\,\,\,\,  x_{ij}^{v}\in \{0,1\}  && \forall v \in \mathcal{V},\forall i,j \in \mathcal{C} \label{eq71}
\end{align} 
The objective function \eqref{objfun} minimizes the total traveled distance of the vehicles. Constraint \eqref{eq211} ensures that each customer is visited by exactly one vehicle. Constraint \eqref{eq311} makes sure that every vehicle is selected once at most. Constraint \eqref{eq41} guarantees the continuity of the route. Constraint \eqref{eq51} ensures that the total demand transported by a vehicle does not exceed its capacity. Constraint \eqref {eq61} prevents the formation of sub-tours and \eqref {eq71} assures the binary nature of the decision variables.

\begin{table}[h!]
	\setlength{\tabcolsep}{3pt} 
	\begin{tabular}{lll}
	\toprule
	& \textbf{Notation}& \textbf{Definition} \\
	\midrule
	{Sets} & \(\mathcal{V}\) & set of vehicles \\
	& \(\mathcal{C}\) & set of customers\\ 
	& \(\mathcal{R}\) & set of routes\\
	& \(\mathcal{S}\) & set of clusters\\ 
	& \(\mathcal{Q}\) & set of demands\\ 
	& \(\Omega\) & set of centroids\\ 
	& \(\mathcal{PR}_l\) & set of pieces of routes starting at the depot\\
	& \(\mathcal{PR}_r\) & set of pieces of routes ending at the depot\\ 
	& \(\mathcal{X}\) & euclidean space\\ 
	& \(\rho_{r}\) & set of customers served by route $r$\\
	& \(prl_{o}\) & piece of route $o$ starting at the depot\\
	& \(prr_{t}\) & piece of route $t$ ending at the depot\\
	\midrule
	{Indices} & $i$, $j$ & customer $i$, $j$ such that $i,j=0$ corresponds to the depot \\ 
	& $v$ & vehicle $v$  \\
	& $r$ & route $r$  \\
	& $k$ & cluster $k$  \\
	& $o$ & piece of route $o$ starting at the depot  \\
	& $t$ & piece of route $t$ ending at the depot  \\
	\midrule
	{Parameters} & $N=|\mathcal{C}|$ & number of customers\\ 
	& $M=|\mathcal{V}|$ & number of vehicles  \\
	& $K=|\mathcal{K}|$ & number of clusters  \\ 
	& $d_{ij}$ & distance between customer $i$ and $j$\\
	& $\delta_{ot}$ & total distance of joining piece of route $o$ and piece of route $t$\\
	& $Q$ & capacity of a vehicle\\
	& $q_i$ & demand of customer $i$, $q_0=0$\\
	& $\tau_{ot}$ & total demand of joining piece of route $o$ and piece of route $t$\\
	& $a_{ir}$ & equal to 1 if customer $i$ is served by route $r$\\
	& $\gamma_{iot}$ & equal to 1 if customer $i$ exists either in the piece of route $o$ or in piece of route $t$\\
	& $c_{r}$ & cost of route $r$\\
	& $\mu_{k}=(\mathscr{x}_{\mu_{k}},\mathscr{y}_{\mu_{k}})$ & euclidean coordinates of centroids $\mu_k$ of cluster $k$\\
	& $\mu_{k}^v=(\mathscr{x}_{\mu_{k}^v},\mathscr{y}_{\mu_{k}^v})$ & euclidean coordinates of centroids $\mu_k^v$ of cluster $k$ yielded by $\mathcal{CVRP}$\\
	& $\mu_{k}^c=(\mathscr{x}_{\mu_{k}^c},\mathscr{y}_{\mu_{k}^c})$ & euclidean coordinates of centroids $\mu_k^c$ of cluster $k$ yielded by $\mathcal{CCBC}$\\
	& $\mathscr{c}_{i}=(\mathscr{x}_{i},\mathscr{y}_{i})$ & euclidean coordinates of customers  $\mathscr{c}_{i}$\\
	& $P=(\mathscr{x}_{P},\mathscr{y}_{P})$ & point from the euclidean space $\mathcal{X}$\\
	\midrule
	Decision variables & $x_{ij}^{v}$ & binary variable equal to 1 if vehicle $v$ visits $j$ after $i$, 0 otherwise\\
	& $\theta_r^v$ & binary variable equal to 1 if route $r$ is used by vehicle $v$, 0 otherwise\\
	& $y_k$ & binary variable equal to 1 if cluster $k$ is selected, 0 otherwise\\
	& $u_o$ & binary variable equal to 1 if piece of route $o$ is selected, 0 otherwise\\
	& $w_t$ & binary variable equal to 1 if piece of route $t$ is selected, 0 otherwise\\
	& $z_{ot}$ & binary variable equal to 1 if piece of route $o$ is joined with piece of route $t$\\
	\bottomrule
	\caption{Mathematical notation.}
	\label{vrpformulation}
\end{tabular}
\end{table}

There exists another common formulation where the $\mathcal{CVRP}$ is introduced as a set partitioning problem ($\mathcal{SPP}$). This formulation is based on the concept of a \textit{route}. We define a \textit{route} as a sequence of customers $\rho_r=\{ \mathscr{c}_0,\mathscr{c}_i,...,\mathscr{c}_{j},\mathscr{c}_0\}$ for $i, j \in \mathcal{C}$ such that $|\rho_r|>2$, and  $\sum_{\mathscr{c}_i \in \rho_r} q_i \leq Q$. For each route $\rho_r$ and customer $\mathscr{c}_i$ such that $r \in \mathcal{R}$ and $ i \in \mathcal{C}$, we define the parameters $a_{ir}$ and $\theta_{r}^{v}$ as follows : 
\begin{equation*}
a_{ir}= \left\{
\begin{array}{ll}
	\ 1& \mbox{if customer } \mathscr{c}_i  \in \rho_r \\
	\ 0 & \mbox{else}
\end{array},
\right.
\text{ and }
\theta_{r}^{v}= \left\{
\begin{array}{ll}
	\ 1& \text{if route $\rho_r$ is used by a vehicle }  v \in \mathcal{V} \\
	\ 0 & \mbox{else}
\end{array}
\right.
\end{equation*}
We define $c_r$ as the travel cost of the route $\rho_r$.
Therefore, one can formulate the set partitioning formulation of the $\mathcal{CVRP}$ as follows:
\begin{align}
	\min & \,\, \sum_{v \in \mathcal{V} }\sum_{r \in \mathcal{R} } c_{r}\theta_{r}^{v} & & \label{ssp-obj}\\
	\text{s.t.} &\,\,\sum_{v \in \mathcal{V}} \sum_{r \in \mathcal{R}} a_{ir} \theta_{r}^{v}=1 && \forall i \in \mathcal{C} \label{eq91}\\
	&\,\, \sum_{v \in \mathcal{V}}\sum_{r \in \mathcal{R}} \theta_{r}^v \leq |\mathcal{V}| && \label{eq101}\\
	&\,\,\theta_{r}^v \in \{0,1\} && \forall r \in\mathcal{R}, v \in\mathcal{V} \label{eq111}
\end{align}

The objective function \eqref{ssp-obj} minimizes the total traveled distance of the vehicles. Constraint \eqref{eq91} guarantees that each customer is visited by exactly one vehicle. Constraint \eqref{eq101} ensures that the targeted number of vehicles is not exceeded. Finally, Constraint \eqref{eq111} makes sure that the decision variables are binary.

\begin{figure}[h!]
	\centering
	\begin{tikzpicture}[scale=1.5]
		\node[draw,fill=red] (depot) at (0,0) {\faHome}; 
		\node[draw,fill=blue] (c1) at (1,2) {\faUser}; 
		\node[draw,fill=blue] (c2) at (3,1) {\faUser}; 
		\node[draw,fill=blue] (c3) at (4,-1) {\faUser}; 
		\node[draw,fill=blue] (c4) at (2,-2) {\faUser}; 
		
		\node[draw,fill=green] (vehicle1) at (1.7,1) {\faTruck}; 
		\node[draw,fill=green] (vehicle2) at (3.4,-0.6) {\faTruck}; 
		
		\draw[dashed,->] (depot) -- (c1);
		\draw[dashed,->] (c1) -- (c2);
		\draw[dashed,->] (c2) -- (depot);
		
		\draw[dashed,->] (depot) -- (c3);
		\draw[dashed,->] (c3) -- (c4);
		\draw[dashed,->] (c4) -- (depot);
		
		\node[above, yshift=3mm] at (c1) {$\mathscr{c}_1$};
		\node[above, yshift=3mm] at (c2) {$\mathscr{c}_2$};
		\node[below, yshift=-3mm] at (c3) {$\mathscr{c}_3$};
		\node[below, yshift=-3mm] at (c4) {$\mathscr{c}_4$};
		
		\node[above right, xshift=2mm, yshift=2mm] at (depot) {Depot};
	\end{tikzpicture}
	\caption{$\mathcal{CVRP}$ problem.}
	\label{cvrp}
\end{figure}

\subsection{Centroid-based clustering in euclidean space}
Centroid-based clustering is an unsupervised machine learning approach to partition data into groups known as clusters with the objective to maximize the similarity within clusters and concurrently minimize similarity between clusters as shown in Figure \ref{cbc}. In other words, in the case of the euclidean distance, it aims at defining each cluster centroid in order to minimize the within-clusters sum of squared distances denoted by \textit{withinss} \cite{wang2011ckmeans}. Among the widely-used algorithms of centroid-based clustering, one can mention \textit{K-means}.

Given a set of customers $\mathcal{C}=\{\mathscr{c}_1, \mathscr{c}_2, ..., \mathscr{c}_n\}$ in the euclidean space denoted by $\mathcal{X}$ where each one is represented by a 2-dimensional vector of euclidean coordinates. We denote $\mu_k$ the centroid of a cluster $S_k \in \mathcal{S}$, such that: $\mu_k=argmin_{x^{'} \in \mathcal{X}}\sum\limits_{\mathscr{c}_i \in S_k} d(\mathscr{c}_i,x^{'})$, and $\mathcal{S}=\{S_1,S_2..,S_K\}$ is the set of $K$ clusters.
One can introduce the optimization formulation for centroid-based clustering as follows:
\begin{align}
	\min        & \,\,  \sum_{k=1 }^{K} \sum_{\mathscr{c}_i \in S_k} d(\mathscr{c}_i,\mu_k) & & \label{eq121}\\
	\text{s.t.} & \,\, \cup_{k=1 }^K S_k= \mathcal{C}& & \label{eq131}\\
	& \,\, S_{k_1} \cap S_{k_2} = \emptyset && \forall S_{k_1}, S_{k_2} \in \mathcal{S}, k_1 \neq k_2\label{eq141}
\end{align}
This mathematical model involves the assignment of the set of customers $\mathcal{C}$ to $K$ clusters while minimizing the \textit{withinss} value which refers to the sum of the euclidean distance $d(\mathscr{c}_i,\mu_k)$ between each $\mathscr{c}_i \in S_k$ and its corresponding centroid $\mu_k$ . This euclidean distance is calculated as follows: 
\begin{equation}
	d(\mathscr{c}_i,\mu_k)=\sqrt{(\mathscr{x}_i-\mathscr{x}_{\mu_k})^2+(\mathscr{y}_i-\mathscr{y}_{\mu_k})^2}
	\label{eq15}
\end{equation}
Each centroid coordinates are computed as follows: 
\begin{equation*}
	\mu_k= (\frac{\sum\limits_{\mathscr{c}_j \in S_k} \mathscr{x}_j}{\left|S_k \right|},\frac{\sum\limits_{\mathscr{c}_j \in S_k} \mathscr{y}_j}{\left|S_k \right|})
\end{equation*}

\begin{figure}[h!]
	\centering
	\begin{tikzpicture}
		\foreach \i/\x/\y in {1/1.5/4, 2/2/3.5, 3/1.8/3.2, 4/2.2/3.8, 5/1.9/3.5} {
			\fill[color=blue] (\x,\y) circle (3pt);
			\node[color=blue] at (\x,\y) {\footnotesize $x_{\i}$};
		}
		
		\foreach \i/\x/\y in {1/4/6, 2/4.5/5.5, 3/4.2/5.8, 4/4.7/6.3} {
			\fill[color=red] (\x,\y) circle (3pt);
			\node[color=red] at (\x,\y) {\footnotesize $x_{\i}$};
		}
		\node[draw,fill=blue,diamond,minimum size=5pt,inner sep=0pt] at (1.9,3.8) {};
		\node[right] at (2.0,4.8) {};
		
		\node[draw,fill=red,diamond, minimum size=5pt, inner sep=0pt] at (4.3,5.9) {};
		\node[above right] at (4.3,5.9) {};
		
		\draw[dashed,blue] (3.3, 4.85) circle (3cm);
		\node[above,blue] at (2.35, 7.9) {\footnotesize $\mathcal{C}$};
		
		\node[above, color=blue] at (2.0, 4.8) {Cluster 1 };
		\node[right, color=red] at (4.3, 5.9) {Cluster 2};
		\node[draw,fill=blue,diamond,minimum size=5pt,inner sep=0pt] at (6.5, 7) {};
		\node[right] at (6.8, 7) {$\mu_1$};
		
		\node[draw,fill=red,diamond,minimum size=5pt,inner sep=0pt] at (6.5, 6.5) {};
		\node[right] at (6.8, 6.5) {$\mu_2$};
		\draw[->] (0,1.5) -- (6,1.5) node[right] {$x$};
		\draw[->] (0,1.5) -- (0,8) node[above] {$y$};
	\end{tikzpicture}
	\caption{Centroid-based clustering.}
	\label{cbc}
\end{figure}
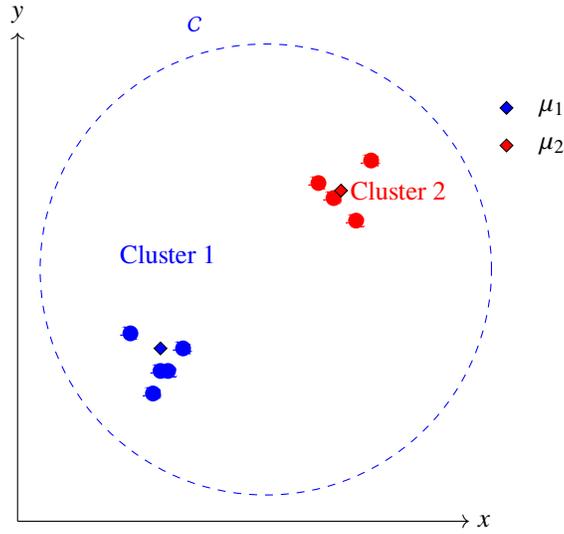

\section{Connection between \texorpdfstring{$\mathcal{CCBC}$ and $\mathcal{CVRP}$}{}} \label{connect-cvrp_ccbc}

The $\mathcal{SPP}$ main objective is to partition elements of a set $S$ into smaller subsets. All items in $S$ must be contained in one and only one subset. The $\mathcal{CCBC}$ problem \eqref{eq121}--\eqref{eq141} can be formulated differently as an $\mathcal{SPP}$ with $K$ subsets, such that $y_k=1$ indicates that $S_k\in \mathcal{S}$ is selected, 0 otherwise. A formal proof can be found in \cite{wu2016clustering} :
\begin{align}
	\min 		& \,\, \sum_{S_k\in \mathcal{C}}\sum_{\mathscr{c}_i \in S_k} d(\mathscr{c}_i,\mu_k) y_k  & & \label{eq17}\\
	\text{s.t.} & \,\, \sum_{S_k: \mathscr{c}_i\in S_k}y_k= 1 && \forall i \in \mathcal{C}\label{eq18}\\
				& \,\, \sum_{S_k \in \mathcal{C}}y_k= K && \label{eq19}\\
				& \,\, y_k \in \{0,1\}  && \forall k \in \mathcal{K}\label{eq20}
\end{align}
One can first observe that the problems $\mathcal{CVRP}$ and $\mathcal{CCBC}$ can both be reduced to a set partitioning problem. Given a set of customers, for fixed centroids the two formulations \eqref{eq121}--\eqref{eq141} and \eqref{eq17}--\eqref{eq20} share the same constraints and differ in terms of objective function. Based on this fact, one can expect defining a formal way to reduce a $\mathcal{CVRP}$ problem to a $\mathcal{CCBC}$. In other words, under some conditions, the exercise of solving a $\mathcal{CCBC}$ provides an optimal or near-optimal solution to $\mathcal{CVRP}$. Subsequently, our objective in what follows is to experimentally elucidate the existence of a connection between $\mathcal{CVRP}$ and $\mathcal{CCBC}$ and deduce some related properties.

\subsection{Exploratory analysis through small-sized examples}
In light of the previous analysis about the connection between $\mathcal{CVRP}$ and $\mathcal{CCBC}$, we aim to assess the extent to which a $\mathcal{CVRP}$ can be reduced to a $\mathcal{CCBC}$. This entails determining whether identifying clusters effectively translates into establishing routes, achieved by implementing a $\mathcal{TSP}$ within each cluster. The following analysis serves as a warm-up study to : 
\begin{enumerate}[label=(\roman*)] 
	\item Highlight through experiments the existence of a connection between these two problems
	\item Deduce and prove some theoretical properties related to this connection between $\mathcal{CVRP}$ and $\mathcal{CCBC}$. 
\end{enumerate}
We particularly focus on elucidating how closely an optimal solution derived from $\mathcal{CCBC}$ problem aligns with an optimal solution of the $\mathcal{CVRP}$. We evaluate this claim through conducting the following  experiment on small-sized instances : 
\begin{itemize}
	\item Randomly generate small-sized instances. Each one encompasses, at most, 10 customers that are spatially distributed within the euclidean space. Concretely, for each instance $I$, a set of $n$ customers is randomly generated from the spatial domain $\left[0,10\right] \times \left[ 0,10\right]$. The demand associated with each of these customers is randomly selected from the interval $\left[ 0,10\right]$. For this problem, the vehicle capacity is 10.
	\item For every instance, we designate a specific customer to be the depot.
	\item We optimally solve $\mathcal{CCBC}$ problem. Every cluster must contain the depot.
	\item Evaluate the derived $\mathcal{CVRP}$  solution through applying a $\mathcal{TSP}$ algorithm within each cluster.
	\item Compare the previous solution with the $\mathcal{CVRP}$ optimal solution.
\end{itemize}

The justification of the use of small-sized instances is grounded in the easiness of achieving optimal solutions for both $\mathcal{CVRP}$ and for $\mathcal{CCBC}$ using exact approaches. Table \ref{t1} illustrates the obtained results. For every size $n$, we conduct the exploratory study described above on 500 instances generated following the aforementioned process. By doing this, we check if $\mathcal{CCBC}$ solution can align with $\mathcal{CVRP}$ solution under some conditions. We denote the set of instances for which the $\mathcal{CCBC}$ optimal solution $Sol_{\mathcal{CCBC}}(I)$ leads to the $\mathcal{CVRP}$ optimal solution $Sol_{\mathcal{CVRP}}(I)$ by $\mathcal{I}_{1}=\{I :  Sol_{\mathcal{CCBC}}(I)\implies Sol_{\mathcal{CVRP}} (I) \}$. Similarly, we refer to the set for which this implication is not verified with  $\mathcal{I}_{2}=\{I  : Sol_{\mathcal{CCBC}}(I) \centernot\implies Sol_{\mathcal{CVRP}} (I) \}$. In details, Table \ref{t1} reports the obtained results in terms of sets size, i.e. $ |\mathcal{I}_{1}|$, $ |\mathcal{I}_{2}|$.

We can clearly notice that for small-size instances the $\mathcal{CCBC}$ optimal clusters coincide with the $\mathcal{CVRP}$ optimal routes in most cases. However, for many other instances they differ in terms of the optimal solution, especially for $n=9$. In light of this outcome, for set $\mathcal{I}_2$, we endeavor to assess the relative gap $\mathcal{GAP}_{r}^{withinss}$ between the value of the optimal $\mathcal{CCBC}$ denoted by $withinss^c(I)$ and the value of the clustering that yields the $\mathcal{CVRP}$ optimal solution  indicated by $withinss^v (I)$. This relative gap is calculated as follows: 
\begin{equation}
	\mathcal{GAP}_{r}^{withinss}(I)=\frac{withinss^v(I)-withinss^c(I)}{withinss^c(I)} \times 100
\end{equation}

Table \ref{tableGap} above confirms that there exists a small relative gap between the optimal $\mathcal{CCBC}$ and the clustering that provides an optimal $\mathcal{CVRP}$ solution in terms of objective function on average $\overline{\mathcal{GAP}_{r}^{withinss}}$. Furthermore, these results grant more legitimacy to the hypothesis claiming that the centroids of the clusters that yield an optimal $\mathcal{CVRP}$ solution are closely located nearby the centroids of clusters derived from the optimal $\mathcal{CCBC}$.

\begin{minipage}{.5\linewidth}
	\begin{table}[H]
		\centering
		\begin{tabular}{c c c } 
			\toprule
			\textit{\textbf{Instance size }} & $ |\mathcal{I}_{1}|$& $ |\mathcal{I}_{2}|$  \\ [0.5ex] 
			\midrule
			$n=5$   & 406& 94\\ 
			$n=7 $ & 248 & 252\\
			$n=9 $ & 161 & 339\\
			\bottomrule
		\end{tabular}
		\caption{ $|\mathcal{I}_{1}|$,  $ |\mathcal{I}_{2}|$ per instance size.}
		\label{t1}
	\end{table}
\end{minipage}
\begin{minipage}{.5\linewidth}
	\begin{table}[H]
		\centering
		\begin{tabular}{c c c } 
			\toprule
			\textit{\textbf{Instance size }} & $\overline{\mathcal{GAP}_{r}^{withinss}}$(\%)\\ [0.5ex] 
			\midrule
			$n=5$  & 1.37\\ 
			$n=7 $ & 0.95\\
			$n=9 $ & 1.70 \\
			\bottomrule
		\end{tabular}
		\caption{ $\overline{\mathcal{GAP}_{r}^{withinss}}$ results per instance size.}
		\label{tableGap}
	\end{table}
\end{minipage}

\paragraph{Delving into a specific example :}
Through the subsequent example, we aim to visually elucidate the previously stated hypothesis and thereby exemplify the connection between $\mathcal{CVRP}$ and $\mathcal{CCBC}$. Concretely, we select one instance $ I^{\ast} \in \mathcal{I}_{2}$ with features introduced in Table \ref{instfeat}. We first represent both $\mathcal{CCBC}$ and $\mathcal{CVRP}$ solutions for this instance in Figures~\ref{fig1} and~\ref{fig2} respectively. We observe that the clusters shapes clearly do not coincide. In particular, customers 2 and 1 belong to cluster 1 within $\mathcal{CCBC}$ optimal solution while customer 2 forms with customers 3 and 4 a single cluster when it comes the $\mathcal{CVRP}$ optimal solution. Consequently, the $\mathcal{CCBC}$ solution does not imply the $\mathcal{CVRP}$ solution because customer 2 is nearer to centroid $\mu_1^c$ rather than $\mu_2^c$. Nevertheless, one can visually notice according to Figure \ref{suppercent} that  $\mathcal{CCBC}$ centroids are near to  $\mathcal{CVRP}$ centroids.
\begin{table}[H]
	\centering
	\begin{tabular}{c c c c} 
		\toprule
		\textit{\textbf{$P$}} & $\mathscr{x}_{P} $& $\mathscr{y}_{P} $ & $q_{P}$ \\ 
		\midrule
		0 & 1& 1 & 0\\
		1 & 2&3&6\\
		2 & 3& 3&1\\
		3 & 2& 5&1\\
		4 & 1& 7&8\\
		\bottomrule
	\end{tabular}
	\caption{ Instance $I^{\ast}$ features.}
	\label{instfeat}
\end{table}

To expand the experimental and the theoretical frame behind the hypothesis, we will conduct an experiment aiming at exploring the neighborhood around the clusters centroids $\{\mu_0^c,\mu_1^c\}$. In detail, for the instance $I^{\ast}$ in Table \ref{instfeat}, we explore around each $\mathcal{CCBC}$ centroid, specifically inside the rectangles defined by:
$\mathcal{RC}_0=\{P=(\mathscr{x}_{P},\mathscr{y}_{P}) \,\, \in \mathcal{X} \,\, \text{:} \,\, min(\mathscr{x}_{\mu_{0}^v},\mathscr{x}_{\mu_{0}^c})\}\leq \mathscr{x}_{P} \leq max(\mathscr{x}_{\mu_{0}^v},\mathscr{x}_{\mu_{0}^c}) \,\,  \text{,} \,\, min(\mathscr{y}_{\mu_{0}^v},\mathscr{y}_{\mu_{0}^c})\}\leq \mathscr{y}_{P} \leq max(\mathscr{y}_{\mu_{0}^v},\mathscr{y}_{\mu_{0}^c})\}$

$\mathcal{RC}_1=\{P=(\mathscr{x}_{P},\mathscr{y}_{P}) \,\, \in \mathcal{X}\,\, \text{:} \,\, min(\mathscr{x}_{\mu_{1}^v},\mathscr{x}_{\mu_{1}^c})\}\leq \mathscr{x}_{P} \leq max(\mathscr{x}_{\mu_{1}^v},\mathscr{x}_{\mu_{1}^c}) \,\, \text{,} \,\, min(\mathscr{y}_{\mu_{1}^v},\mathscr{y}_{\mu_{1}^c})\}\leq \mathscr{y}_{P} \leq max(\mathscr{y}_{\mu_{1}^v},\mathscr{y}_{\mu_{1}^c})\}$

We search for centroids candidates that provide the optimal solution for $\mathcal{CVRP}$. The exploration findings are presented in Figure \ref{explocent}. It is clear that around each centroid within the explored neighborhood, we can define a region of multiple points that gives an optimal solution of the $\mathcal{CVRP}$. 
These findings provide more credibility to the conjecture which states that it is sufficient to solve a $\mathcal{CCBC}$ when solving a $\mathcal{CVRP}$. Explicitly, one can first solve the $\mathcal{CCBC}$ and simultaneously  evaluate the corresponding $\mathcal{CVRP}$ solution, then search for a better centroids combination within the neighborhood that improves the $\mathcal{CVRP}$ solution. 
From a practical point of view, one can target the nearest points to the $\mathcal{CCBC}$ centroids that guarantee an optimal solution of the $\mathcal{CVRP}$.

In the context of the example we are examining, we can formulate the task of finding the nearest points combination to the $\mathcal{CCBC}$ centroids that provides a $\mathcal{CVRP}$ optimal solution as a quadratic multivariate mathematical model \eqref{eq21}--\eqref{eq25}:
\begin{align}
	\min_{ \mu_1, \mu_2 \in \mathcal{X}} & \,\,\,\, (\mathscr{x}_{\mu_1}- \mathscr{x}_{\mu_1^c})^2+(\mathscr{y}_{\mu_1}- \mathscr{x}_{\mu_1^c})^2+(\mathscr{x}_{\mu_2}- \mathscr{x}_{\mu_2^c})^2+(\mathscr{y}_{\mu_2}- \mathscr{x}_{\mu_2^c})^2
	\label{eq21}\\
	s.t &\,\,\,\,(\mathscr{x}_{\mu_1}- \mathscr{x}_1)^2+(\mathscr{y}_{\mu_1}- \mathscr{y}_1)^2 \leq (\mathscr{x}_{\mu_2}- \mathscr{x}_1)^2+(\mathscr{y}_{\mu_2}- \mathscr{y}_1)^2
	\label{eq22}\\
	&\,\,\,\,  (\mathscr{x}_{\mu_2}- \mathscr{x}_2)^2+(\mathscr{y}_{\mu_2}- \mathscr{y}_2)^2 \leq (\mathscr{x}_{\mu_1}- \mathscr{x}_2)^2+(\mathscr{y}_{\mu_1}- \mathscr{y}_2)^2
	\label{eq23}\\
	&\,\,\,\, (\mathscr{x}_{\mu_2}- \mathscr{x}_3)^2+(\mathscr{y}_{\mu_2}- \mathscr{y}_3)^2\leq (\mathscr{x}_{\mu_1}- \mathscr{x}_3)^2+(\mathscr{y}_{\mu_1}- \mathscr{y}_3)^2
	\label{eq24} \\
	&\,\,\,\,(\mathscr{x}_{\mu_2}- \mathscr{x}_4)^2+(\mathscr{y}_{\mu_2}- \mathscr{y}_4)^2 \leq (\mathscr{x}_{\mu_1}- \mathscr{x}_4)^2+(\mathscr{y}_{\mu_1}- \mathscr{y}_4)^2 
	\label{eq25}
\end{align}

The objective function, as defined in Equation \eqref{eq22}, aims to identify centroids $(\mu_1,\mu_2)$ that are closer to the $\mathcal{CCBC}$ centroids $(\mu_1^c, \mu_2^c)$. Constraints \eqref{eq23}--\eqref{eq26}  are designed to ensure that the identified centroids match the customers assignment to clusters as given by the $\mathcal{CCBC}$ outcome. Specifically, customer 1 is assigned to the cluster centered around $\mu_1$, while customers 2, 3, and 4 are grouped in the cluster around $\mu_2$. It should be noted that Constraints \eqref{eq23}--\eqref{eq26} aim to minimize the distance between each cluster and its corresponding centroid, thereby reducing the overall $withinss$ as introduced in Equation \eqref{eq121}.

We use a sequential least square programming  solver $\mathcal{SLSP}$ to get a solution of the model above \eqref{eq21}--\eqref{eq25}. We plot the corresponding centroids in Figure~\ref{NearestCCbc}. An observation reveals that there exist centroids closer to the $\mathcal{CCBC}$ ones, which enables obtaining the optimal solution for the $\mathcal{CVRP}$.

\begin{figure}[h!]
	\centering
	\begin{minipage}{0.48\textwidth}
		\centering
		\includegraphics[width=\linewidth]{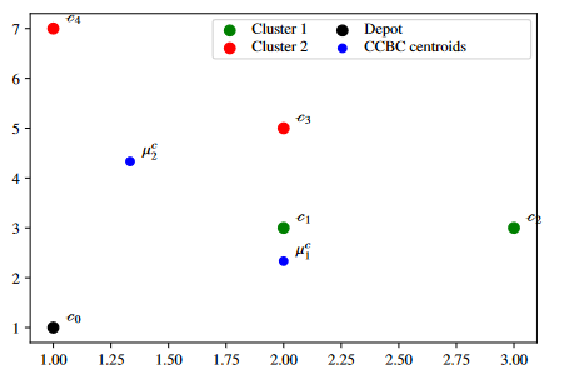} 
		\caption{Optimal $\mathcal{CCBC}$.}
		\label{fig1}
	\end{minipage}\hfill
	\begin{minipage}{0.48\textwidth}
		\centering
		\includegraphics[width=\linewidth]{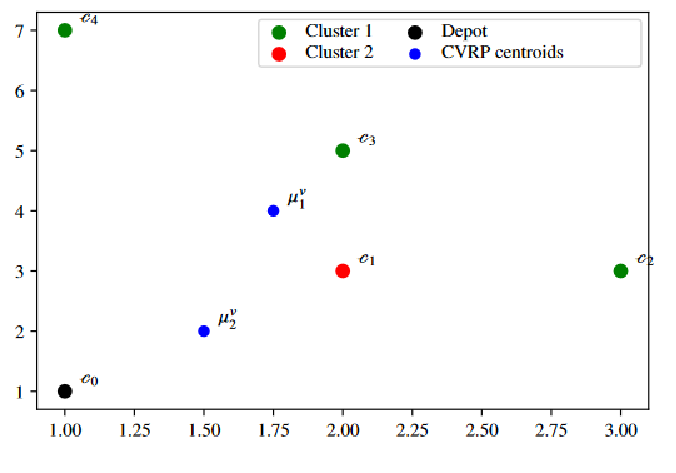} 
		\caption{Optimal $\mathcal{CVRP}$.}
		\label{fig2}
	\end{minipage}
\end{figure}

\begin{figure}[h!]
	\centering
	\begin{minipage}{0.48\textwidth}
		\centering
		\includegraphics[width=\linewidth]{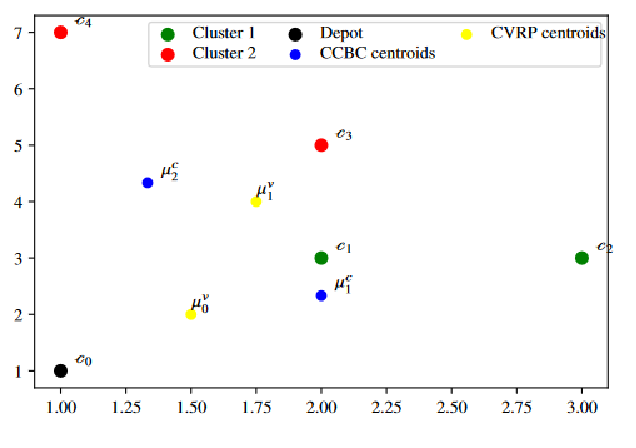} 
		\caption{Supperposition of $\mathcal{CCBC}$ and $\mathcal{CVRP}$ centroids.}
		\label{suppercent}
	\end{minipage}\hfill
	\begin{minipage}{0.48\textwidth}
		\centering
		\includegraphics[width=\linewidth]{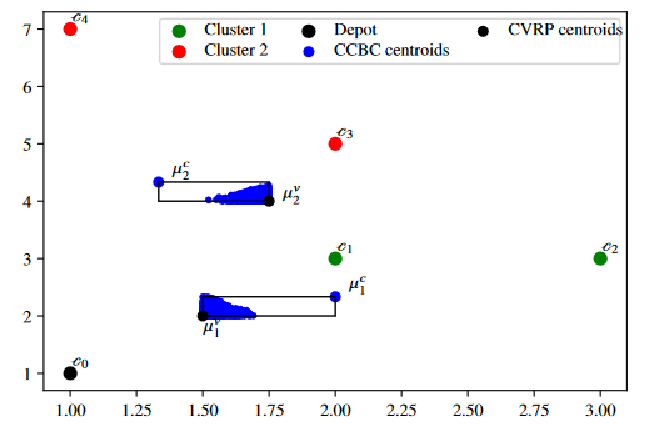} 
		\caption{Exploring nearest $\mathcal{CCBC}$ centroids.}
		\label{explocent}
	\end{minipage}
	\begin{minipage}{0.48\textwidth}
		\centering
		\includegraphics[width=\linewidth]{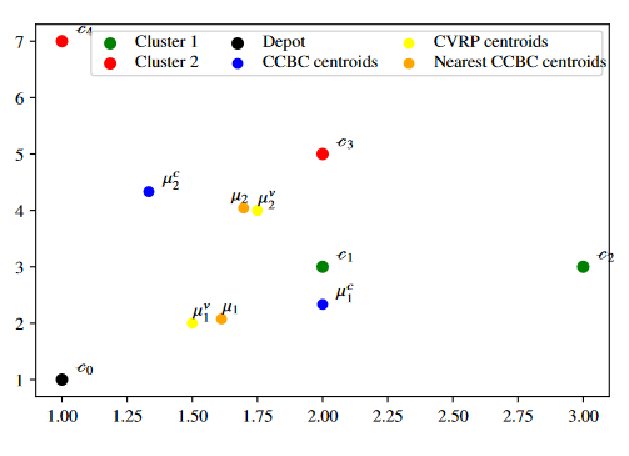} 
		\caption{Nearest $\mathcal{CCBC}$ centroids.}
		\label{NearestCCbc}
	\end{minipage}
\end{figure}


\subsection{Generalization of the connection between \texorpdfstring{$\mathcal{CCBC}$ and $\mathcal{CVRP}$}{}}
The idea behind carrying out this experimental study is to show that the $\mathcal{CVRP}$ solution is reachable when selecting the centroids from a specific region in the space. In general, knowing the optimal routes shape, we can formulate the task of finding the nearest centroids combination from $\mathcal{CCBC}$ ones that lead to this optimal $\mathcal{CVRP}$ solution~\eqref{eq26}--\eqref{eq28}. It should be noted that each $S_k \in \mathcal{S}$ encompasses customers from the route $\rho_k$ within the $\mathcal{CVRP}$ optimal solution, i.e., $S_k=\{\mathscr{c}_i : \mathscr{c}_i \in \rho_k \}$. 
\begin{align}
	\min_{\mathscr{x}_{\mu_k},\mathscr{y}_{\mu_k} \in \mathcal{X}} & \,\, \sum_{k=1}^K (\mathscr{x}_{\mu_k}- \mathscr{x}_{\mu_k^c})^2+(\mathscr{y}_{\mu_k}- \mathscr{y}_{\mu_k^c})^2 & & \label{eq26}\\
	\text{s.t.} & \,\,(\mathscr{x}_{\mu_k}- \mathscr{x}_j)^2+(\mathscr{y}_{\mu_k}- \mathscr{y}_j)^2 \leq \min_{k' \in \mathcal{K}-\{k\}}(\mathscr{x}_{\mu_{k'}} - \mathscr{x}_j)^2+(\mathscr{y}_{\mu_{k'}}- \mathscr{y}_j)^2 && \forall \mathscr{c}_j  \in \mathcal{S}_k, \forall  k  \in \mathcal{K}\label{eq27}\\
	& \,\, \mathscr{x}_{\mu_k},\mathscr{y}_{\mu_k} \geq 0 &&\forall  k \in  \mathcal{K}  \label{eq28}
\end{align}

In what follows, we introduce a theoretical characterization of the centroids regions that yield an optimal solution for the $\mathcal{CVRP}$.
\begin{definition}
	In the context of a $\mathcal{CCBC}$ problem, a \textit{strict centroid} refers to a centroid $\mu_{k^{\ast}}$ of cluster $ S_{ k^{\ast}} \in \mathcal{S}=\{S_1,S_2, \ldots, S_K\}$, such that:
	\begin{align}
		\left\{
		\begin{array}{ll}
			\ d(\mu_{k^{\ast}},\mathscr{c}_i)< d(\mu_{k},\mathscr{c}_i) \,\,\,\,\, \forall \mathscr{c}_i  \in S_{k^{\ast}},\forall  k  \in \mathcal{K}-\{k^{\ast}\}  \\
			\   d(\mu_{k^{\ast}},\mathscr{c}_j)> d(\mu_{k},\mathscr{c}_j) \,\,\,\,\, \forall  \mathscr{c}_j  \in S_{k},\forall k  \in \mathcal{K}-\{k^{\ast}\}
		\end{array}
		\label{eq31}
		\right.
	\end{align}
	This implies that there are no customers located at an identical euclidean distance from this centroid and any other centroid.
\end{definition}

\begin{remark}
	As the incoming theoretical results rely on this particular definition, we aimed to assess the occurrence of strict centroids within $\mathcal{CVRP}$ instances. To achieve this, we conduct a straightforward experiment using benchmark instances from groups \textit{A, B, P, E} presented in Section \ref{compres}. Concretely, the developped procedure involves the following steps:
	\begin{itemize}
		\item For each instance from the aforementioned groups, where the optimal solution is known we determine centroids combination that yields this solution. In other words, centroids combination that satisfies Equations \eqref{eq26}--\eqref{eq28}.
		\item We check for every centroid within the combination if Equation \eqref{eq31} is verified.  
	\end{itemize}
	As a result of this study, we can confirm that for every instance, all the found centroids can be denoted as strict.
	
\end{remark}
\begin{theorem}
	\label{pythagorean}
	In the context of the $\mathcal{CVRP}$, if there exists a set of centroids combination $\Omega = \{\mu_1, \mu_2, \ldots,$\break $\mu_K \}$ that yields an optimal $\mathcal{CVRP}$ solution and includes a \textit{strict centroid} $\mu_{k^{\ast}}$, then there exists an infinite number of centroids combinations that also provide an optimal solution for the $\mathcal{CVRP}$.
\end{theorem}
\begin{proof}
	Let's denote a centroids combination $\Omega = \{\mu_1, \ldots, \mu_{k^{\ast}}, \ldots,\mu_K \}$ that gives an optimal $\mathcal{CVRP}$ solution for the customers set $\mathcal{C}$ and includes a strict centroid $\mu_{k^{\ast}}$.
	Let's prove that: $\exists \,\,\mu_{k^{\ast}}^{\alpha}=(\mathscr{x}_{\mu_{k^{\ast}}^{\alpha}},\mathscr{y}_{\mu_{k^{\ast}}^{\alpha}}) $ near to $\mu_{k^{\ast}}=(\mathscr{x}_{\mu_{k^{\ast}}},\mathscr{y}_{\mu_{k^{\ast}}}) $ such that $\Omega = \{\mu_1, \ldots, \mu_{k^{\ast}}^{\alpha}, \ldots,\mu_K \}$ guarantees an optimal solution as well for the $\mathcal{CVRP}$. 
	
	To demonstrate this, we need to prove that:
	\begin{align}
		(\mathscr{x}_{\mu_{k^{\ast}}^{\alpha}}- \mathscr{x}_i)^2+(\mathscr{y}_{\mu_{k^{\ast}}^{\alpha}}- \mathscr{y}_i)^2 &\leq \min_{k \in \mathcal{K}-\{k^{\ast}\}}(\mathscr{x}_{\mu_k}- \mathscr{x}_i)^2+(\mathscr{y}_{\mu_k}- \mathscr{y}_i)^2 && \forall \mathscr{c}_i  \in \mathcal{S}_{k^{\ast}} 
		\label{32}
	\end{align}
	And 
	\begin{align}
		(\mathscr{x}_{\mu_k}- \mathscr{x}_j)^2+(\mathscr{y}_{\mu_k}- \mathscr{y}_j)^2 &\leq (\mathscr{x}_{\mu_{k^{\ast}}^{\alpha}}- \mathscr{x}_j)^2+(\mathscr{y}_{\mu_{k^{\ast}}^{\alpha}}- \mathscr{y}_j)^2 
		&& \forall \mathscr{c}_j  \in  \mathcal{S}_k, \forall k \in \mathcal{K}-\{k^{\ast}\}\label{33}
	\end{align}
	We define for $\mathscr{c}_i=(\mathscr{x}_i,\mathscr{y}_i), i \in \mathcal{C} $ and for $k \in K $: 
	\begin{align*}
		F_i(\mathscr{x}_{\mu_k},\mathscr{y}_{\mu_k})=(\mathscr{x}_{\mu_k}- \mathscr{x}_i)^2+(\mathscr{y}_{\mu_k}- \mathscr{y}_i)^2
	\end{align*}
	Let's put : 
	\begin{align*}
		\beta= \min_{\mathscr{c}_j \in \mathcal{S}_k,  k  \in \mathcal{K}-\{k^{\ast} \}} 
		\{  F_j(\mathscr{x}_{\mu_{k^{\ast}}},\mathscr{y}_{\mu_{k^{\ast}}}) -F_j(\mathscr{x}_{\mu_k},\mathscr{y}_{\mu_k}) \} > 0 \,\,\,\,\,\,\,\,(\text{strict centroids})
	\end{align*}
	Let's put as well :
	\begin{align*}
		f, e=argmin_{\mathscr{c}_j \in \mathcal{S}_k,  k  \in \mathcal{K}-\{k^{\ast} \}} 
		\{  F_j(\mathscr{x}_{\mu_{k^{\ast}}},\mathscr{y}_{\mu_{k^{\ast}}}) -F_j(\mathscr{x}_{\mu_k},\mathscr{y}_{\mu_k}) \} 
		\label{eq2}
	\end{align*}
	Let's choose point $ \mu_{k^{\ast}}^{\alpha}=(\mathscr{x}_{\mu_{k^{\ast}}^{\alpha}},\mathscr{y}_{\mu_{k^{\ast}}^{\alpha}})$ between $\mu_{k^{\ast}}=(\mathscr{x}_{\mu_{k^{\ast}}},\mathscr{y}_{\mu_{k^{\ast}}})$ and $\mu_e=(\mathscr{x}_{\mu_e},\mathscr{y}_{\mu_e})$ such that the distance between $\mu_{k^{\ast}}^{\alpha}$ and $\mu_{k^{\ast}}$ lower than $\beta$ as illustrated in Figure  \ref{choose}, for the point $\mathscr{c}_f \in \mathcal{S}_{e}$:
	\begin{figure}[H]
		\centering
		\begin{tikzpicture}
			\coordinate (P1) at (0,0); 
			\coordinate (if) at (2.4,2.4); 
			\coordinate (Pe) at (2,0.7); 
			\def\theta{1.5} 
			\draw[dashed,red] (P1) circle (\theta) (P1) -- ++(80:\theta) node[midway, above right] {$\beta$};
			
			\fill (P1) circle (2pt);
			
			\fill (if) circle (2pt);
			
			\fill (Pe) circle (2pt);
			
			\draw[blue] (if) -- (P1) node[midway, below,left] {$\tilde{b}$};
			\draw (if) -- (Pe);
			\draw (Pe) -- (P1);
			
			\node[below left] at (P1) {$\mu_{k^{\ast}}$};
			\node[above right] at (if) {$\mathscr{c}_f$};
			\node[below] at (Pe) {$\mu_e$};
			
			\coordinate (P1alpha) at ($(Pe)!0.5!(P1)$);
			
			\fill (P1alpha) circle (2pt);
			\draw[blue] (P1alpha) -- (if) node[midway, below] {$\tilde{c}$};
			\draw[blue] (P1alpha) -- (P1) node[midway, above] {$\tilde{a}$};
			
			\node[below] at (P1alpha) {$\mu_{k^{\ast}}^{\alpha}$};
			
		\end{tikzpicture}
		\caption{Choosing point $\mu_{k^{\ast}}^{\alpha}=(\mathscr{x}_{\mu_{k^{\ast}}^{\alpha}},\mathscr{y}_{\mu_{k^{\ast}}^{\alpha}})$.}
		\label{choose}
	\end{figure}
	Consequently, we define :
	\begin{align*} 
		\mathscr{x}_{\mu_{k^{\ast}}^{\alpha}}= \alpha \mathscr{x}_{\mu_{k^{\ast}}} + (1-\alpha) \mathscr{x}_{\mu_e}\quad  \text{and} \quad  \mathscr{y}_{\mu_{k^{\ast}}^{\alpha}}= \alpha \mathscr{y}_{\mu_{k^{\ast}}} + (1-\alpha) \mathscr{y}_{\mu_e} \quad \text{such that } \alpha \in \left]0, 1\right[\
	\end{align*}
	We can choose any $\mu_{k^{\ast}}^{\alpha}=(\mathscr{x}_{\mu_{k^{\ast}}^{\alpha}},\mathscr{y}_{\mu_{k^{\ast}}^{\alpha}})$ inside the circle varying $\alpha \in \left]0, 1\right[$. For $\mathscr{c}_i \in \mathcal{S}_{k^{\ast}}$, we define: 
	\begin{equation*}
		\eta_i= F_i(\mathscr{x}_{\mu_e},\mathscr{y}_{\mu_e})- F_i(\mathscr{x}_{\mu_{k^{\ast}}},\mathscr{y}_{\mu_{k^{\ast}}}) > 0 \textit{et }  \zeta_i= \min_{k\in \mathcal{K}-\{k^{\ast}\}}\left[(\mathscr{x}_{\mu_k}- \mathscr{x}_i)^2+(\mathscr{y}_{\mu_k}- \mathscr{y}_i)^2\right]- F_i(\mathscr{x}_{\mu_{k^{\ast}}},\mathscr{y}_{\mu_{k^{\ast}}}) > 0
	\end{equation*}
	We choose $\mu_{k^{\ast}}^{\alpha}=(\mathscr{x}_{\mu_{k^{\ast}}^{\alpha}},\mathscr{x}_{\mu_{k^{\ast}}^{\alpha}})$ such that:
	\begin{align}
		l_i= (1-\alpha) \eta_i - \zeta_i &\leq 0 && \forall \mathscr{c}_i \in \mathcal{S}_{k^{\ast}} \label{eqf}
	\end{align}
	\begin{remark}
		The value of $l_i$ can be always negative by increasing $\alpha \in \left]0, 1\right[$, because 
		$\lim\limits_{\alpha\to 1^{-}} l_i= -\zeta_i <0$.
		Then, the point $\mu_{k^{\ast}}^{\alpha}=(\mathscr{x}_{\mu_{k^{\ast}}^{\alpha}},\mathscr{y}_{\mu_{k^{\ast}}^{\alpha}})$ becomes nearer to $\mu_{k^{\ast}}=(\mathscr{x}_{\mu_{k^{\ast}}},\mathscr{y}_{\mu_{k^{\ast}}})$, but it always remains inside the circle $(\mu_{k^{\ast}},\beta)$ in Figure \ref{choose}.
	\end{remark}
	It is clear according to Figure \ref{choose} and triangle inequality that :
	\begin{align*}
		\tilde{b}\leq \tilde{a}+\tilde{c} \implies \tilde{b}-\tilde{a}\leq \tilde{c} \implies \tilde{b}- \beta \leq \tilde{c}  \implies (\mathscr{x}_{\mu_e}- \mathscr{x}_f)^2+(\mathscr{x}_{\mu_e}- \mathscr{y}_f)^2 \leq (\mathscr{x}_{\mu_{k^{\ast}}^{\alpha}}- \mathscr{x}_f)^2+(\mathscr{y}_{\mu_{k^{\ast}}^{\alpha}}- \mathscr{y}_f)^2   
	\end{align*}
	Using the same process we can deduce that :
	\begin{align*}
		(\mathscr{x}_{\mu_k}- \mathscr{x}_j)^2+(\mathscr{x}_{\mu_k}- \mathscr{y}_j)^2 &\leq (\mathscr{x}_{\mu_{k^{\ast}}^{\alpha}}- \mathscr{x}_j)^2+(\mathscr{y}_{\mu_{k^{\ast}}^{\alpha}}- \mathscr{y}_j)^2   && \forall \mathscr{c}_j\in  S_k, \forall k \in \mathcal{K}-\{k^{\ast}\}
	\end{align*}
	Hence Equation \eqref{33} is verified, and let's prove Equation \eqref{32}. \\
  For $\mathscr{c}_i \in S_{k^{\ast}}$, we know that $F_i$ is a convex bi-variate function, because its hessian is :
	\begin{align*}
		\nabla^2F_i(\mathscr{x}_{\mu_k},\mathscr{y}_{\mu_k})& =
		\begin{pmatrix}
			2 & 0 \\
			0 & 2 
		\end{pmatrix} && \forall i \in \mathcal{C}
	\end{align*}
	Hence, 
	\begin{align*}
		&F_i(\alpha \mathscr{x}_{\mu_{k^{\ast}}}+ (1-\alpha) \mathscr{x}_{\mu_e},\alpha \mathscr{y}_{\mu_{k^{\ast}}} + (1-\alpha) \mathscr{y}_{\mu_e}) \leq \alpha F_i(\mathscr{x}_{\mu_{k^{\ast}}},\mathscr{y}_{\mu_{k^{\ast}}} ) + (1-\alpha) F_i(\mathscr{x}_{\mu_e},\mathscr{x}_{\mu_e}) \\
		&\implies F_i(\mathscr{x}_{\mu_{k^{\ast}}^{\alpha}},\mathscr{y}_{\mu_{k^{\ast}}^{\alpha}}) \leq \alpha F_i(\mathscr{x}_{\mu_{k^{\ast}}},\mathscr{y}_{\mu_{k^{\ast}}} ) + (1-\alpha) F_i(\mathscr{x}_{\mu_{k^{\ast}}},\mathscr{y}_{\mu_{k^{\ast}}} ) + (1-\alpha) \eta_i\\
		&\implies F_i(\mathscr{x}_{\mu_{k^{\ast}}^{\alpha}},\mathscr{y}_{\mu_{k^{\ast}}^{\alpha}}) \leq  F_i(\mathscr{x}_{\mu_{k^{\ast}}},\mathscr{y}_{\mu_{k^{\ast}}}) + (1-\alpha) \eta_i \\
		&\implies F_i(\mathscr{x}_{\mu_{k^{\ast}}^{\alpha}},\mathscr{y}_{\mu_{k^{\ast}}^{\alpha}})  \leq \min_{k\in \mathcal{K}-\{k^{\ast}\}}(\mathscr{x}_{\mu_k}- \mathscr{x}_i)^2+(\mathscr{y}_{\mu_k}- \mathscr{y}_i)^2 - \zeta_i + (1-\alpha) \eta_i \\
		&\implies F_i(\mathscr{x}_{\mu_{k^{\ast}}^{\alpha}},\mathscr{y}_{\mu_{k^{\ast}}^{\alpha}}) \leq \min_{k \in \mathcal{K}-\{k^{\ast}\}}(\mathscr{x}_{\mu_k}- \mathscr{x}_i)^2+(\mathscr{y}_{\mu_k}- \mathscr{y}_i)^2 \,\,\,\,\,\,\,\,\,\,\,\,\,\,\,\,\,\,\,\,\,\,\,\,\,\,\,\, ((1-\alpha) \eta_i - \zeta_i \leq 0 \text{ according to $\eqref{eqf}$})
	\end{align*}
		{Hence, Equation \eqref{32} is verified as well. Finally, one can deduce that $\Omega = \{\mu_1, \ldots, \mu_{k^{\ast}}^{\alpha}, \ldots,\mu_K \}$ guarantees an optimal solution as well. Then there exists an infinite number of centroids combinations that also provide an optimal solution for the $\mathcal{CVRP}$.}
	\end{proof}

	\paragraph{Graphical representaion :} Let's  plot graphically this sub-region, we define $\alpha_{0}$ such that: 
	\begin{equation}
		l_j= (1-\alpha_{0}) \eta_j - \zeta_j=0
	\end{equation}
	Let's consider $\mu_{k^{\ast}}^{\alpha_0}= (\mathscr{x}_{\mu_{k^{\ast}}^{\alpha_0}},\mathscr{x}_{\mu_{k^{\ast}}^{\alpha_0}})$, such that : 
	\begin{align*} 
		\mathscr{x}_{\mu_{k^{\ast}}^{\alpha_0}}= \alpha_0 \mathscr{x}_{\mu_{k^{\ast}}} + (1-\alpha_0) \mathscr{x}_{\mu_e} \quad \text{and} \quad \mathscr{y}_{\mu_{k^{\ast}}^{\alpha_0}}= \alpha_0 \mathscr{y}_{\mu_{k^{\ast}}} + (1-\alpha_0) \mathscr{y}_{\mu_e} \quad \text{such that } \alpha_0 \in \left[0, 1\right]\
	\end{align*}
	
	Let's denote the distance between points $\mu_{k^{\ast}}$ and $\mu_{k^{\ast}}^{\alpha_0}$ as $\psi_{\alpha_0}$. So we have two cases to choose point~$\mu_{k^{\ast}}^{\alpha}$:
	\begin{align*}
		\left
		\{
		\begin{array}{ll}
			\ \mu_{k^{\ast}}^{\alpha} \in \textbf{C}(\mu_{k^{\ast}},\psi_{\alpha_0})& \text{if} \,\,\,\psi_{\alpha_0} \leq \beta  \\
			\  \mu_{k^{\ast}}^{\alpha} \in \textbf{C}(\mu_{k^{\ast}},\beta) & \mbox{otherwise}
		\end{array}
		\right.
	\end{align*}
	
	Such that $\textbf{C}(O,\mathscr{r})$ is the circle of center $O$ and radius $\mathscr{r}$. Here is below the graphical representation if $\psi_{\alpha_0} \leq \beta$ in Figure \ref{finregion}. The  point $\mu_{k^{\ast}}^{\alpha}$ can be selected within the gray region.

	\begin{figure}[H]
		\centering
		\begin{tikzpicture}
			\coordinate (P1) at (0,0); 
			\coordinate (pj) at (2.4,2.4); 
			\coordinate (Pe) at (2,0.7);
			\def\theta{1.5} 
			\def\dalpha{1} 
			\def\opacity{0.5} 
			\draw[dashed,red] (P1) circle (\theta) (P1) -- ++(60:\theta) node[midway, above right] {$\beta$};
			
			\fill (P1) circle (2pt);
			
			\fill (pj) circle (2pt);
			
			\fill (Pe) circle (2pt);
			
			\draw[blue] (pj) -- (P1) node[midway, below,left] {$\tilde{b}$};
			\draw (pj) -- (Pe);
			\draw (Pe) -- (P1);
			\node[below left] at (P1) {$\mu_{k^{\ast}}$};
			\node[above right] at (pj) {$\mathscr{c}_f$};
			\node[below] at (Pe) {$\mu_e$};
			\coordinate (P1alpha0) at ($(Pe)!0.5!(P1)$);
			\fill (P1alpha0) circle (2pt);
			\draw[blue] (P1alpha0) -- (pj) node[midway, below] {$\tilde{c}$};
			\draw[blue] (P1alpha0) -- (P1) node[midway, above] {$\tilde{a}$};
			\node[below] at (P1alpha0) {$\mu_{k^{\ast}}^{\alpha_0}$};
			\draw[dashed,blue,fill=gray,fill opacity=\opacity] (P1) circle (\dalpha) (P1) -- ++(110:\dalpha) node[midway, above left] {$\psi_{\alpha_0}$};
		\end{tikzpicture}
		\caption{The region of $\mu_{k^{\ast}}^{\alpha}$ if $\psi_{\alpha_0} \leq \beta$.}
		\label{finregion}
	\end{figure}
	
	We can proceed similarly to define a region around each strict centroid within the combination $\Omega$. It is noteworthy that this proof characterizes only a subset of the region that yields an optimal $\mathcal{CVRP}$ solution. However, there may exist other, closer sub-regions that lead to equivalent outcomes.

		\section{Solution methodology} \label{solmeth}
		
		Given the highlighted connection between $\mathcal{CVRP}$ and $\mathcal{CCBC}$, our goal is to design a $\mathcal{CCBC}$ based framework for addressing $\mathcal{CVRP}$. In this section we first present a general overview of the proposed approach in Subsection \ref{general-overview}. Subsequently, we shed light on the methodological aspects of each component within the advocated framework, including the constrained centroid-based clustering, optimization, and re-optimization in Subsections~\ref{methodCCBC} and~\ref{methodopt}
		
		\subsection{Methodology overview}\label{general-overview}

		The proposed methodology is an extension of the \textit{cluster-first, route-second} heuristic for tackling the $\mathcal{CVRP}$. More specifically, we endow the original framework with improvements to ensure better quality and runtime results by leveraging the aforementioned connection between $\mathcal{CVRP}$ and $\mathcal{CCBC}$. In contrast to the well-known \textit{cluster-first, route-second}, the proposed approach includes three steps as highlighted in Flowchart \ref{framework}. Moreover, we label the proposed approach as \textit{Cluster \& Tune First, Route Second, Ruin \& Recreate Third } and abbreviate it as $\mathcal{CTR}3$.
\begin{figure} [h!]\centering
	\includegraphics*[width=.6\linewidth]{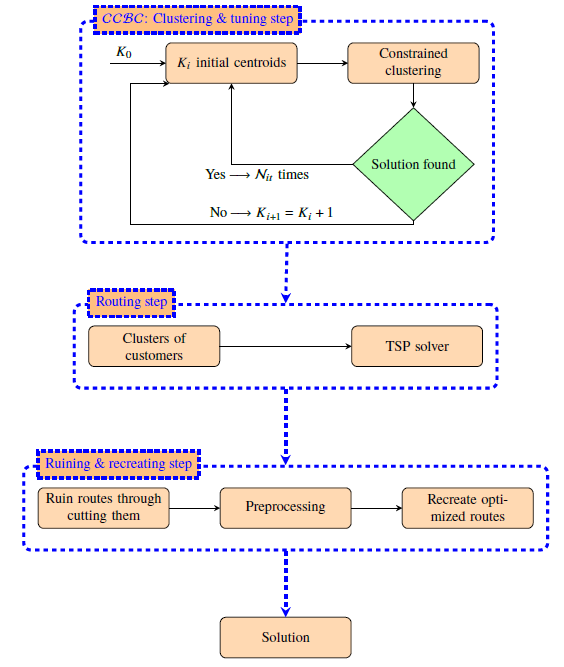}
	\caption{Flow chart of the $\mathcal{CTR}3$ framework.}\label{framework} 
\end{figure}

	\paragraph{Clustering \& tuning step : } The $\mathcal{CCBC}$ algorithm takes the customers raw data, including euclidean coordinates and demands as an input. Then it proceeds to partitioning customers into $K$ clusters. This step is carried out $\mathcal{N}_{it}$ times varying the initial centroids combination through a random multi-start procedure. This iterative approach serves a dual purpose:  first, it alleviates the local optimum impact, thereby ensuring better clustering results, and second, it aims at leveraging the connection between $\mathcal{CVRP}$ and $\mathcal{CCBC}$ by reaching centroids regions that can provide a $\mathcal{CVRP}$ near-optimal solution. Furthermore, the customers are assigned to clusters subject to the capacity constraint and based on a customized metric that prioritizes customers who are nearby and have a high demand first. Besides, the proposed clustering algorithm starts with a lower bound for the number $K$ of clusters and it is self-adjusted if no solution is found taking into account 
the capacity constraint. More details will be provided in Section \ref{methodCCBC}.

\paragraph{Routing step : } This second phase involves designing routes by ordering customers within each cluster, considering that each vehicle must begin and end at the depot. This task is handled using a $\mathcal{TSP}$ algorithm.

\paragraph{Ruining \& recreating step : } This third step aims at re-optimizing the routes using a cutting and relinking approach to design the final routes by means of an integer linear program. In brief, this phase is primarily focused on defining two sets of routes pieces using a route cutting process and subsequently combining elements from these sets using a relinking procedure to establish final routes. The nomenclature for this step draws inspiration from the \textit{Ruin and Recreate} algorithm. Further elaboration will be given in Section \ref{methodopt}.

		\subsection{\texorpdfstring{$\mathcal{CCBC}$:}{} Clustering and tuning step } \label{methodCCBC}

		We opt for a constrained centroid-based clustering, primarily due to the reason we outlined previously through the connection between $\mathcal{CCBC}$ and $\mathcal{CVRP}$. This latter highlights the feasibility of attaining an optimal or near-optimal solution by appropriately selecting centroids combination within specific regions.
		Furthermore, an additional justification arises from the intrinsic characteristics of the $\mathcal{CVRP}$ where the predefined lower bound for the number of clusters $K$ simplifies the clustering process. The proposed clustering approach keeps the core structure of the well-known \textit{k-means} alongside incorporating the needed adjustments to take into account:
		\begin{itemize}
			\item Self-adjustment of the total number of clusters $K$.
			\item Centroids multi-start initialization procedure.
			\item Assignment metric.
			\item Capacity constraint.
		\end{itemize}
		In the context of a constrained clustering, the designed clusters must meet a set of constraints. These are commonly known in the literature as two sets, namely: \textit{Must-link} constraints and \textit{Cannot-link} constraints \cite{basu2008constrained}. To elaborate, a \textit{must-link} constraint involving the tuple $(m,p)$ requires that $m$ is in cluster $S_i$ if and only if $p$ is in cluster $S_i$. In contrast, a \textit{cannot-link} constraint for the tuple $(m,p)$ stipulates that $m$ is in $S_i$ if and only if $p$ is not in $S_i$. In our particular case, a \textit{Cannot-link} constraint is implemented within the $\mathcal{CCBC}$ algorithm \ref{IR}. It explicitly addresses the impracticality of associating a set of customers within the same cluster such that the total demand exceeds the vehicle capacity. More accurately, the algorithm looks over the capacity constraint violation during each assignment operation and then acts accordingly. It should be emphasized that in the context of the $\mathcal{CCBC}$ algorithm the clusters are shaped with the objective of minimizing the \textit{withinss}. This latter is defined in our context as follows:
		\begin{equation*}
			\min \sum_{k=1}^K \sum_{\mathscr{c}_i \in S_k} d(\mathscr{c}_i, \mu_k^c )^2  \text{ such that }
			\mu_k^c= (\frac{\sum_{\mathscr{c}_i \in S_k} \mathscr{x}_i}{\left|S_k \right|},\frac{\sum_{\mathscr{c}_i \in S_k} y_i}{\left|S_k \right|}) \text{ and } \sum_{\mathscr{c}_i \in S_k} q_i \leq Q  \quad \forall  S_k  \subseteq \mathcal{C}
		\end{equation*}
		
		In the case of a homogenous fleet, the number of targeted clusters is lower bounded by the following value:
		\begin{equation*}
			K_0=\lceil \frac{\sum_{i=1}^{N} q_i}{Q}\rceil
		\end{equation*}
		Furthermore, Algorithm \ref{IR} starts the clustering with the predetermined lower bound. If it fails while designing the clusters, the number of required clusters is increased iteratively as indicated in the chart flow \ref{framework}. On the top of that, we rely on a customized assignment metric to cluster different customers. This latter is calculated for every customer with respect to each cluster by means of the following formula:
		\begin{equation}
			\mathcal{AM}(\mathscr{c}_i,S_k)=\frac{q_{i} }{d( \mathscr{c}_i,\mu_k^c)}
			\label{equa11}
		\end{equation}
		This coefficient prioritizes assigning near customers with high demand to the clusters. To be more specific, the clusters are primarily populated with customers whose assignment metric is greater. Whenever the vehicle capacity is reached, the remaining customers are subsequently assigned to the second nearest cluster in order. As marked out in Algorithm \ref{IR}, the clustering methodology is enriched with a multi-start process to start up the algorithm through randomly selecting centroids combination at each iteration. This procedure improves the clustering solution by mitigating the local optima impact and exploring the solution space with the goal of reaching centroids regions that can provide $\mathcal{CVRP}$ near-optimal solution as discussed in Section \ref{connect-cvrp_ccbc}. The clustering step yields feasible and unordered routes with respect to the capacity constraint. The designed algorithm generates $\mathcal{N}_{it}$ solutions as presented in Algorithm \ref{IR}. Each of these solutions corresponds to a specific starting centroids combination. Following that, the routing stage in flow chart \ref{framework} consists of applying a $\mathcal{TSP}$ solver within every cluster to provide valid routes starting from and ending at the depot. The $\mathcal{TSP}$ solver is obviously applied on every cluster within each generated solution among the $\mathcal{N}_{it}$ solutions. One can choose the best one that provides the minimum traveled distance. However, in our case, the proposed approach makes use of all $\mathcal{CCBC}$ solutions in the re-optimization step, as described in Section~\ref{methodopt}.
		
		\subsection{Ruining \& recreating step}\label{methodopt}
		
		As previousely stated, it is intractably hard to find the optimal solution for the $\mathcal{CCBC}$ taking into account the capacity constraint, particularly as the number of customers increases. Therefore the clustering approach explores various combinations of centroids to approach a near-optimal clustering result. Due to the heuristic nature of this process, it probably generates some inaccurately clustered customers. This fact paves the way to further improvements for the $\mathcal{CVRP}$ global solution. This latter depends directly on the customers assignment to clusters. A single misplaced customer can completely change the final solution.
		
		To minimize the effect of this issue, our methodology incorporates a ruining and recreating step. This approach phase makes use of the routes obtained directly after applying the $\mathcal{TSP}$ solver. Concretely, this step consists, on the one hand, in cutting every route into a pair of pieces and on the other hand, designing new routes by relinking the pieces of the routes in an optimal way. The ruining operation can be handled using a function that iterates each route and returns all the possible pieces through cutting in the different route edges. Consequently, there are many cutting configurations for a given route depending on the cutting position. Finally, recreating optimized routes can be crafted by relinking pieces of routes using an integer linear program to select the tuples to match. This task can be modeled visually as an assignment problem with constraints in Figure \ref{assign}. Following the problem description, we can formulate it using an MILP~\eqref{eq35}--\eqref{eq40}.
		\begin{align}
			\min & \,\,\sum_{o \in \mathcal{PR}_l}\sum_{t \in \mathcal{PR}_r }\delta_{ot}z_{ot} & &  \label{eq35}\\
			\text{s.t} &\,\,\sum_{o \in \mathcal{PR}_l}\sum_{t \in \mathcal{PR}_r}\gamma_{iot} z_{ot}=1 && \forall i\in\mathcal{C}\label{eq36}\\
			&\,\, \sum_{o \in \mathcal{PR}_l}z_{ot} =w_{t} && \forall t \in\mathcal{PR}_r \label{eq37}\\
			&\,\, \sum_{t \in \mathcal{PR}_r}z_{ot} =u_{o} && \forall  o \in\mathcal{PR}_l
			\label{eq38}\\
			&\,\, \tau_{ot} z_{ot}\leq Q && \forall  o \in \mathcal{PR}_l,  \forall t\in \mathcal{PR}_r
			\label{eq39}\\
			&\,\, z_{ot} \geq 0 ; u_{o},w_{t} \in \{0,1\} && \forall o \in \mathcal{PR}_l,  \forall  t\in \mathcal{PR}_r \label{eq40}
		\end{align}
	
	\begin{algorithm}[h!]\footnotesize
		\DontPrintSemicolon
		\KwData{$\mathcal{C}, Q, \mathcal{N}_{it}, G$ such that $G$ is a gap limit}
		\KwResult{$\Omega=(\mu_1,...\mu_K)$ with $S_k=\{\mathscr{c}_1,..,\mathscr{c}_{n_k}\}$}
		\Begin{
			$K_0 \longleftarrow \frac{\sum_{i=1 }^{N}q_i }{C}$, \,\, $i=1$.\;
			
			\For{$ i \leq N_{it}$}{
				$\Omega_{\mu}=\{\mu_1,\mu_2,...,\mu_K\} \longleftarrow Random(\mathcal{C},K)\,\,\,\,\,$ \tcp{ $Random$ is function randomly selects $K$ elements from customers set $\mathcal{C}$.}\;
				{ $gap \longleftarrow+ \inf $ }\;
				\nl\While{$gap \geq G$}{
					$ \mathcal{DM}\longleftarrow Distance Matrix(\mathcal{C},\Omega_{\mu}, euclidean) \,\,\,\,\,$ \tcp{ $Distance Matrix$ is a function returning an array of dimension $N\times K$ such that the element $(i,k)$ is the euclidean distance between $\mathscr{c}_i$ and centroid $\mu_k$.}\;
					\For{$S_k \in \mathcal{S}$}{ Decreasingly ordering $I_{k}=\{\mathscr{c}_1,..\mathscr{c}_{k'}\}$, set of points nearest to $S_{k}$ according to the priority $\mathcal{AM}(\mathscr{c}_j,S_{k}) \longleftarrow \frac{q_{j}}{\mathcal{DM}\,(j, k) }$\;
						{$\overset{\star}{C}\longleftarrow 0$      \,\,\tcp{$\overset{\star}{C}$ is the consumed capacity }\; }
						\For{$\mathscr{c}_i \in I_{k}$}{ 
							{$\overset{\star}{C} \longleftarrow \overset{\star}{C} + q_{i} $\; }
							\eIf{$\overset{\star}{C} < Q$}{ Assign $\mathscr{c}_i$ to $S_{k}$
							}{
								Assign $\mathscr{c}_i$ to $I_{k'}$, such that $S_{k'}$ is the nearest cluster to $\mathscr{c}_i$ after $S_{k}$\;
						}} 
					} Update centroids $\Omega$: $\mu_k\longleftarrow \frac{1}{\lvert S_k\rvert} \sum_{\mathscr{c}_j \in S_k} \mathscr{c}_j$ \\ $gap \longleftarrow UpdateGap (\Omega, S)$ \tcp{ $UpdateGap$ is a function that calculates $withness$ for current solution and than $gap$} }
				Return $\Omega$\\
				$i \longleftarrow i+1$ 
			}
		}
	\caption{$\mathcal{CCBC}$ Algorithm.}\label{IR}
\end{algorithm}

\begin{figure}[h!]
	\begin{tikzpicture}[x=2cm, y=2cm,scale=0.8, spy using outlines={circle, magnification=3, size=2cm, connect spies}]
		\tikzstyle{every node}=[draw, circle, fill=gray, inner sep=3pt,minimum size=18pt] 
		
		\draw (-0.5,0) node[minimum size=1.5cm](Output){Depot};
		\draw (7.5,0) node[minimum size=1.5cm](Input){Depot};
		\draw (2,2) node(1){$u_1$};
		\draw (2,1.5) node(2){$u_2$};
		\draw (2,1) node(3){$u_3$};
		\draw (2,0.5) node(4){$u_4$};
		\draw (2,0) node(5){$\dots$};
		\draw (2,-0.5) node(6){$\dots$};
		\draw (2,-1.5) node(7){$\dots$};
		\draw (2,-2) node(8){$u_n$};
		\draw (5,1.5) node(9){$w_1$};
		\draw (5,1) node(10){$w_2$};
		\draw (5,0.5) node(11){$w_3$};
		\draw (5,0) node(12){$w_4$};
		\draw (5,-0.5) node(13){$\dots$};
		\draw (5,-1) node(14) {$\dots$};
		\draw (5,-1.5) node(15){$w_m$};
		
		\draw[->] (Output) -- (1);
		\draw[->] (Output) -- (2);
		\draw[->] (Output) -- (3);
		\draw[->] (Output) -- (4);
		\draw[->] (Output) -- (5);
		\draw[->] (Output) -- (6);
		\draw[->] (Output) -- (7);
		\draw[->] (Output) -- (8);
		\draw[->] (1) -- (9);
		\draw[->] (1) -- (11);
		\draw[->] (1) -- (15);
		\draw[->] (2) -- (10);
		\draw[->] (2) -- (12);
		\draw[->] (3) -- (13);
		\draw[->] (4) -- (14);
		\draw[->] (5) -- (15);
		\draw[->] (6) -- (9);
		\draw[->] (6) -- (14);
		\draw[->] (7) -- (13);
		\draw[->] (8) -- (9);
		\draw[->] (8) -- (15);
		\draw[->] (9) -- (Input);
		\draw[->] (10) -- (Input);
		\draw[->] (11) -- (Input);
		\draw[->] (12) -- (Input);
		\draw[->] (13) -- (Input);
		\draw[->] (14) -- (Input);
	\end{tikzpicture}
	\caption{Pieces of routes assignment problem.} \label{assign}
\end{figure}
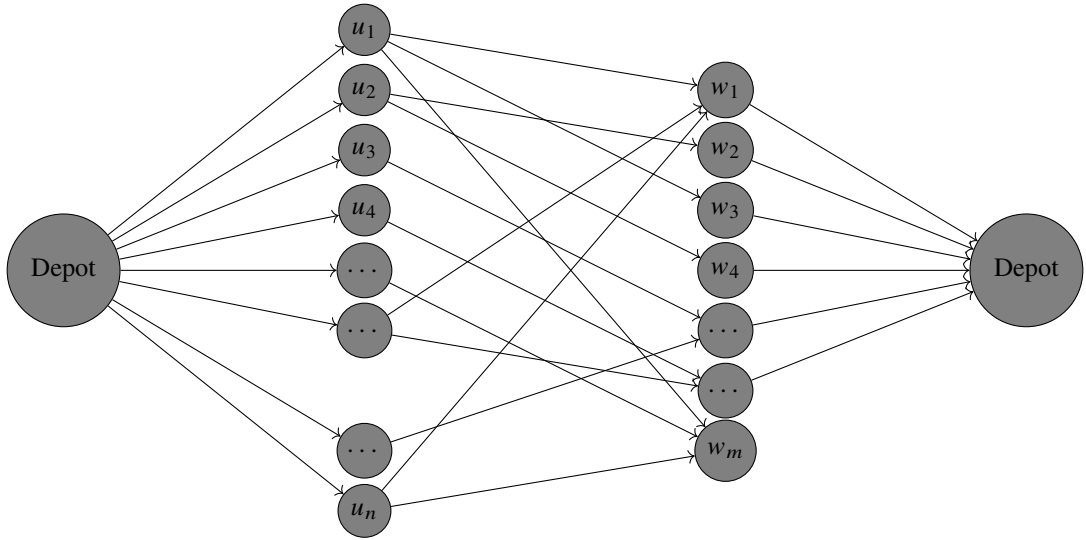

The objective function \eqref{eq35} minimizes the total traveled distance. Constraint \eqref{eq36} guarantees that each customer is visited exactly once. Constraints \eqref{eq37}, \eqref{eq38} are the flow conservation constraints. Finally, Constraint~\eqref{eq40} establishes the variables nature.

For the ruin \& recreate step, we apply some specific preprocessing strategies to reduce the runtime required by the MILP above to get an optimal solution. Explicitly, it mainly aims at network reduction by removing the infeasible edges. In our case,  an infeasible edge is triggered in the following cases: 
	\begin{itemize}
		\item \textbf{Common client: } this constraint consists in eliminating arcs if it matches two pieces of routes that share a common client:
		\begin{equation}
			prl_{o} \cap prr_{t} \neq \emptyset \implies z_{ot}=0
		\end{equation}
		\item   \textbf{Exceeding the vehicle capacity: } eliminate an arc if the total demand of two matched pieces of routes exceed the truck capacity.
		\begin{equation}
			\tau_{o}+\tau_{t} > Q \implies z_{ot}=0
		\end{equation}
		
		\item \textbf{Exceeding the global capacity gap : } we define here two notions, namely: the global capacity gap $\mathcal{CG}_{glob}$, and the arc capacity gap
		$\mathcal{CG}_{ot}$ for the pieces of routes $(prl_{o},prr_{t})$.
		\begin{enumerate}
			\item $\mathcal{CG}_{glob}$ : this means the gap between the total demand and the total capacity of all used vehicles.
			\item $\mathcal{CG}_{ot}$ : this refers to the gap between the two linked pieces of routes total demand $(o,t)$ and vehicle capacity.
		\end{enumerate}
		Consequently, this remark eliminates arcs according to the following equation:
		\begin{equation}
			\mathcal{CG}_{ot}> \mathcal{CG}_{glob} \implies z_{ot}=0
		\end{equation}
		such that, 
		\begin{equation}
			\mathcal{CG}_{glob} = K\times Q - \sum_{i=1}^N q_i \quad \textbf{and} \quad \mathcal{CG}_{ot}=Q - (\tau_{o}+\tau_{t})
		\end{equation}
		As mentioned above $K$ is the number of clusters, $D$ refers to the customer's total demand, $Q$ is the vehicle capacity.
	\end{itemize}

\section{Experimentation} \label{exper}
In this section, we present the computational experiments conducted on $\mathcal{CVRP}$ known instances in the literature. Explicitly, in Subsection \ref{compres}, we introduce the test plan and  report the computational results that arise from this experimentation. These experiments are compared to the best-known solutions in the literature \cite{lysgaard2004new}. More than that, we extend the benchmark to include the results from\cite{ewbank2019capacitated} and \cite{ewbank2016unsupervised} since they rely on a similar methodology. The subsequent Subsections  \ref{multi-start-impact}, \ref{am-impact}, and \ref{reopt-impact} are dedicated to a post-computational analysis  with a primary focus on elucidating the effectiveness of different components within the proposed framework.

\subsection{Computational results} \label{compres}
As stated before, the targeted experiments involve existing $\mathcal{CVRP}$ instances in the literature, namely: groups $A$, $B$ and $P$ from \cite{augerat1995computational}, and group $E$ from \cite{christofides1969algorithm}. It is noteworthy that these instances are small and medium-sized and each one is uniquely identified by a nomenclature convention denoted as \textit{$G-nx_1-kx_2$}, which means a $\mathcal{CVRP}$ instance from group $G$ with $x_1$ customers and the optimal solution corresponds to $x_2$ vehicles.

The proposed framework is implemented using Python. All experiments are carried out on a 3.20GHz Intel(R) Core(TM) i7-8700
processor, with 64GiB System memory, using a Linux operating system. The Integer Linear Program is solved using the IBM CPLEX
Commercial Solver (version 12.9.0.0). We use the Pulp library (version 2.7.0) to communicate with CPLEX solver from Python.

We compare the proposed approach results to the baselines approaches from the literature, namely \cite{lysgaard2004new,ewbank2019capacitated}, and \cite{ewbank2016unsupervised} using the following metrics : relative gap denoted by $\mathcal{GAP}_r$, runtime, and number of optimal solutions obtained per approach $\mathcal{N}^{opt}$. The relative gap for instance $I$ is calculated with regard to the best-known solution, following the formula introduced below :
\begin{equation}
	\mathcal{GAP}_r(I)=\frac{\mathcal{SOL}(I)-\mathcal{BN}(I)}{\mathcal{BN}(I)} \times 100
\end{equation}
Such that, $\mathcal{SOL}(I)$ is the value of the $\mathcal{CVRP}$ solution for instance $I$ using the proposed approach and $\mathcal{BN}(I)$ is the value of the best-known solution for the same instance. Table \ref{A} thoroughly reports the results obtained for groups $A$, $B$, $P$, $E$ using the proposed  approach alongside with the ones provided by the work of \cite{ewbank2019capacitated,ewbank2016unsupervised}. For all these three approaches, we calculate $\mathcal{GAP}_r(I)$ with regard to the best known solution $\mathcal{BN}(I)$ \cite{lysgaard2004new}. Table \ref{recap-gap} sums up these results average for each approach.  Notably, the proposed methodology significantly outperforms other baselines, achieving an average relative gap $\overline{\mathcal{GAP}_r}$ of 1.07\% with respect to the best-known solution. 

We can derive alternative comparison insights concerning the three approaches based on the number of optimal solutions $\mathcal{N}_{opt}$ reached by each one as introduced in Table \ref{recap-gap}. Considering this new metric, the proposed approach considerably outperforms the other ones as the optimum is attained for 9 instances. This fact confirms the relevance of the additional techniques we incorporate to improve $\mathcal{CCBC}$ framework, namely : multi-start initial centroids, assignment metric, re-optimization procedure. The impact of this improvements strategies will be elaborated in Subsections \ref{multi-start-impact}, \ref{am-impact}, and \ref{reopt-impact}.

\begin{table}[h!]
	\begin{tabular}{l l l} 
		\toprule
		\textit{\textbf{Method}}& $\overline{\mathcal{GAP}_r}$ (\%)& $\mathcal{N}^{opt}$ \\
		\midrule
		$\mathcal{CTR}3$ & 1.07&9 \\
		\textit{\textbf{Ewbank and al. \cite{ewbank2019capacitated}}} & 1.58&1 \\
		\textit{\textbf{Ewbank and al. \cite{ewbank2016unsupervised}}} &3.45&2  \\
		\bottomrule
	\end{tabular}
	\caption{ $\overline{\mathcal{GAP}_r}$ and $\mathcal{N}^{opt}$ reached by each approach.}
	\label{recap-gap}
\end{table}

For better analysis, we convert this numerical comparison to visual plots. Observing Figure \ref{distplotgap}, one can get an idea about the relative gap distribution among the three approaches. The one related to the proposed approach is approximately similar to a Gaussian curve. More than that, it looks more symmetric and concentrated around the mean, unlike other approaches. The box plot in Figure \ref{boxplotgap} provides an additional perspective on the performance of these three approaches. In detail, our approach outperforms the other ones in terms of the upper quartile. In contrast to the proposed approach, one can notice that the remaining approaches can lead to some isolated $\mathcal{GAP}_r$ values according to their box plot. According to this analysis, one can conclude that our methodology presents less variability in terms of $\mathcal{GAP}_r$. This characteristic holds an advantageous practical value in real-world applications.
Furthermore, we can point out according to Table \ref{table-runtime} that the proposed approach provides $\mathcal{CVRP}$ solutions within reasonable timeframe, specifically, a maximum of 240 seconds to solve \textit{E-n76-k7}. These runtime values using $\mathcal{CTR}3$ are competitive with those of \cite{ewbank2019capacitated,ewbank2016unsupervised}.


\begin{figure}[H]
	\centering
	 \includegraphics[width=0.9\linewidth]{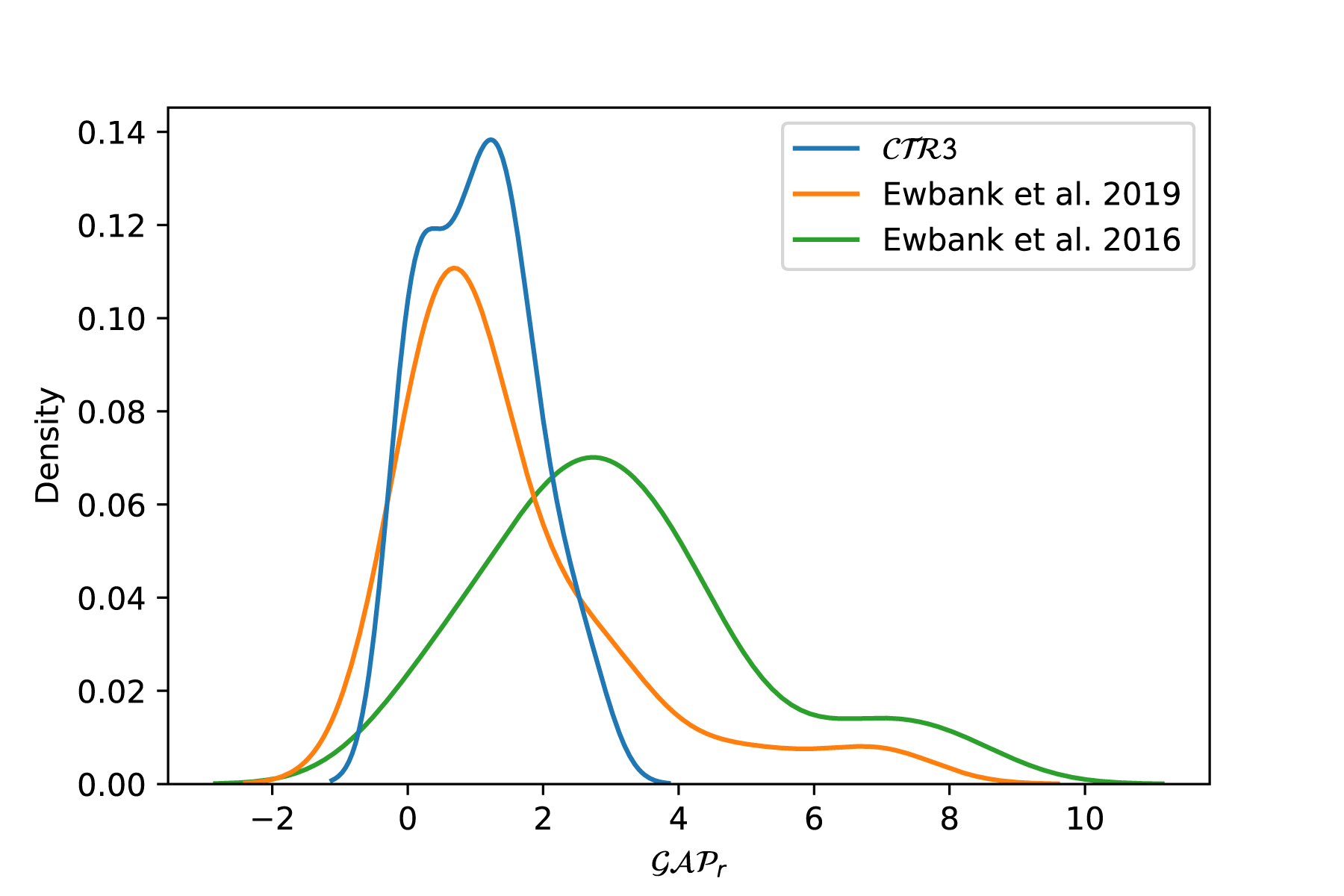}
	\caption{Distribution plot of $\mathcal{GAP}_r$ for the three approaches}
	\label{distplotgap}
\end{figure}
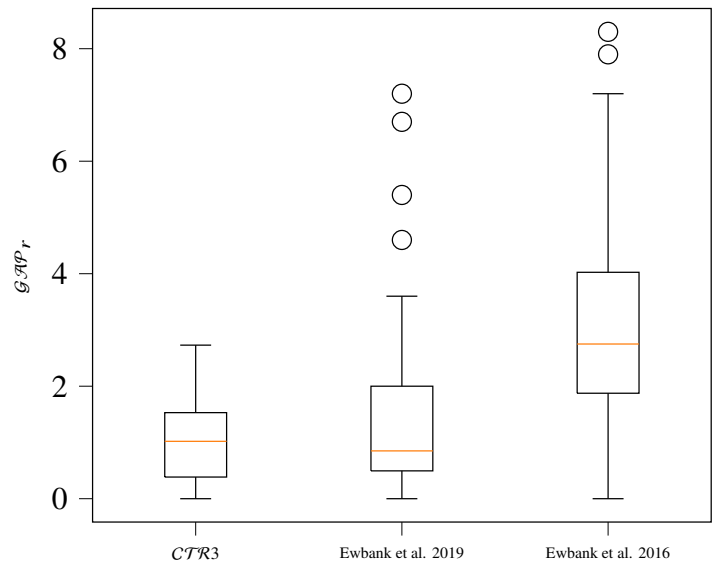
\begin{figure}[H]
	\begin{tikzpicture}[scale=1.2]
		
		\definecolor{darkgray176}{RGB}{176,176,176}
		\definecolor{darkorange25512714}{RGB}{255,127,14}
		
		\begin{axis}[
			tick align=outside,
			tick pos=left,
			title={},
			x grid style={darkgray176},
			xmin=0.5, xmax=3.5,
			xtick style={color=black},
			xtick={1,2,3},
			xticklabels={\tiny{$\mathcal{CTR}3$},\tiny{Ewbank~{et al.} 2019},\tiny{Ewbank~{et al.} 2016}},
			y grid style={darkgray176},
			ylabel={\tiny{$\mathcal{GAP}_r$}},
			ymin=-0.415, ymax=8.715,
			ytick style={color=black}
			]
			\addplot [black]
			table {%
				0.85 0.385
				1.15 0.385
				1.15 1.53
				0.85 1.53
				0.85 0.385
			};
			\addplot [black]
			table {%
				1 0.385
				1 0
			};
			\addplot [black]
			table {%
				1 1.53
				1 2.73
			};
			\addplot [black]
			table {%
				0.925 0
				1.075 0
			};
			\addplot [black]
			table {%
				0.925 2.73
				1.075 2.73
			};
			\addplot [black]
			table {%
				1.85 0.495
				2.15 0.495
				2.15 2
				1.85 2
				1.85 0.495
			};
			\addplot [black]
			table {%
				2 0.495
				2 0
			};
			\addplot [black]
			table {%
				2 2
				2 3.6
			};
			\addplot [black]
			table {%
				1.925 0
				2.075 0
			};
			\addplot [black]
			table {%
				1.925 3.6
				2.075 3.6
			};
			\addplot [black, mark=o, mark size=3, mark options={solid,fill opacity=0}, only marks]
			table {%
				2 5.4
				2 7.2
				2 6.7
				2 4.6
			};
			\addplot [black]
			table {%
				2.85 1.875
				3.15 1.875
				3.15 4.025
				2.85 4.025
				2.85 1.875
			};
			\addplot [black]
			table {%
				3 1.875
				3 0
			};
			\addplot [black]
			table {%
				3 4.025
				3 7.2
			};
			\addplot [black]
			table {%
				2.925 0
				3.075 0
			};
			\addplot [black]
			table {%
				2.925 7.2
				3.075 7.2
			};
			\addplot [black, mark=o, mark size=3, mark options={solid,fill opacity=0}, only marks]
			table {%
				3 7.9
				3 8.3
			};
			\addplot [darkorange25512714]
			table {%
				0.85 1.02
				1.15 1.02
			};
			\addplot [darkorange25512714]
			table {%
				1.85 0.85
				2.15 0.85
			};
			\addplot [darkorange25512714]
			table {%
				2.85 2.75
				3.15 2.75
			};
		\end{axis}
		
	\end{tikzpicture}
	\caption{Boxplot of $\mathcal{GAP}_r$ for the three approaches.}
	\label{boxplotgap}
\end{figure}

\clearpage
\newpage
\begin{sidewaystable}[h!]
	\centering
	\resizebox{\textwidth}{!}{
		\begin{tabular}{cc cc cc cc cc cc cc cc} 
			\toprule
			\multicolumn{1}{c}{Instance} & 
			\multicolumn{2}{c}{$\mathcal{CTR}3$} & 
			\multicolumn{2}{c}{Ewbank~{et~al.}  2019~\cite{ewbank2019capacitated}} &
			 \multicolumn{2}{c}{Ewbank~{et~al.} 2016~\cite{ewbank2016unsupervised}} &
			\multicolumn{1}{C{1.3cm}}{Lysgaard {et~al.}~\cite{lysgaard2004new}}&
			\multicolumn{1}{c}{Instance}& 
			\multicolumn{2}{c}{$\mathcal{CTR}3$}& 
			\multicolumn{2}{c}{Ewbank~{et al.} 2019 \cite{ewbank2019capacitated}}& 
			\multicolumn{2}{c}{Ewbank~{et al.} 2016 \cite{ewbank2016unsupervised}} & 
			\multicolumn{1}{C{1.3cm}}{Lysgaard {et~al.}~\cite{lysgaard2004new}}
			\\
			\cmidrule(rl){2-3} \cmidrule(rl){4-5} \cmidrule(rl){6-7}\cmidrule(rl){10-11}\cmidrule(rl){12-13}\cmidrule(rl){12-13} \cmidrule(rl){14-15}
			& {$\mathcal{SOL} (I)$} & {$\mathcal{GAP}_r$ (\%)} &  {$\mathcal{SOL} (I)$} & {$\mathcal{GAP}_r$ (\%)}&{$\mathcal{SOL} (I)$} & $\mathcal{GAP}_r$ (\%) & & &{$\mathcal{SOL} (I)$}& $\mathcal{GAP}_r$ (\%) & {$\mathcal{SOL} (I)$} & $\mathcal{GAP}_r$ (\%)& {$\mathcal{SOL} (I)$} & $\mathcal{GAP}_r$ (\%)& \\
			\midrule
			A-n32-k5 &791& 0.88& 787& 0.4& 812& 3.6   &784& B-n41-k6 &  833 & 0.48& 834 &0.6 &851 &2.7&  829  \\ 
			A-n33-k5 &674  &1.92& 662& 0.2 &680& 2.9 & 661&B-n43-k6 & 754& 1.59& 748 &0.8 &767& 3.4  & 742 \\
			A-n33-k6 & 745 &0.40& 745& 0.4 &756 &1.9 & 742& B-n44-k7 & 917 & 0.87& 914& 0.6 &928& 2.1  &	909\\
			A-n34-k5 & 788  & 1.26& 786& 1.0& 789& 1.4 &778 & B-n50-k7 & 747 & 0.80& 744& 0.4 &794& 7.2  & 741\\
			A-n36-k5& 808  & 1.11& 802& 0.4& 816& 2.1 & 799 & B-n52-k7 & 753& 0.79 & 752& 0.7 & 766 & 2.5 & 747\\
			A-n37-k5 & 679 & 1.47&672 &0.4& 693& 3.6 & 669& B-n63-k10 &1538 & 2.73& 1506& 0.7& 1566& 4.7  & 1496\\ 
			A-n37-k6 & 963 & 1.45&  1.00& 5.4& 976& 2.8 & 949&E-n22-k4 	& 375&0.00& 402& 7.2& 375& 0.0 &375 \\
			A-n38-k5& 746 &  2.14& 750& 2.7& 747&2.3 & 730& E-n23-k3& 569  &  0.00& 569& 0.0 &569& 0.0 &569\\
			A-n39-k5&833 & 1.32& 829 &0.9 &849& 3.3 & 822 & E-n30-k3&557 &4.12& 539 &0.9 &576 &7.9 &534\\
			A-n39-k6 & 845&  1.65&835& 0.5 &845 &1.7 &  831&E-n33-k4& 846&1.30& 837& 0.2 &845& 1.2 &835\\ 
			A-n44-k6 & 937&0.00& 1000& 6.7 &964 &2.9  &937& E-n51-k5&521 & 0.00& 535 &2.7& 524& 0.6 &521\\ 
			A-n48-k7 & 1091 & 1.64& 1083& 0.9& 1098 &2.3 &1073& E-n76-k7&696& 2.01& 692& 1.5& 721& 5.7 &682\\
			A-n53-k7 &  1023 &  1.27&  1021 &1.1& 1094 & 8.3 &  1010&P-n20-k2 & 216 & 0.00 & 217& 0.5& 218& 0.9 & 216\\ 
			A-n54-k7 & 1181 & 1.18& 1188 & 1.8& 1198 &2.7  & 1167&P-n21-k2 & 211 & 0.00 & 217& 2.8 &219& 3.8 & 211\\
			A-n63-k10 & 1320 & 0.45& 1329 &1.1 &1367& 4.0  &1314&P-n22-k2 & 216  & 0.00& 217& 0.5 &217 &0.5  & 216  \\ 
			A-n64-k9 & 1422 &1.47& 1438 &2.6 &1460 & 4.2  &1401 &P-n40-k5& 458& 0.00& 461& 0.7& 468 &2.2  & 458\\ 
			A-n69-k9 &	1171 & 1.02& 1179& 1.7& 1209& 4.3 & 1159&P-n45-k5 	& 510 & 0.00& 513 &0.6 &512 &0.4  & 510\\ 
			A-n80-k10 &1780&	0.95& 1794&  1.8 & 1877 &6.5  &1763&P-n50-k7 	& 560& 1.07& 560 &1.1& 579& 4.5  & 554\\ 
			B-n34-k5 & 789 & 0.12& 793 &0.6& 802& 1.8  & 788&P-n55-k10 & 698 & 0.57& 700& 4.6 &716&7.0 & 694  \\
			B-n35-k5 & 968 & 1.34&  956& 0.1 &980& 2.6    & 955&P-n76-k4 	& 608 & 2.40 & 610 &3.6 &605& 2.7& 593\\
			B-n38-k6 &  808 & 0.37& 808& 0.4& 835 &3.7 & 805&P-n76-k5 & 641 & 2.18& 644& 2.7& 636& 1.4 & 627\\
			B-n39-k5 & 557 & 1.43& 553& 0.7& 568& 3.5& 549&P-n101-k4 & 693 & 1.73& 702 &3.1 &709 &4.1 &681 \\	
			\bottomrule
		\end{tabular}
	}
	\caption{Comparison between $\mathcal{CTR}3$ results and baselines from the literature in terms of $\mathcal{GAP}_r$.}
	\label{A}
\end{sidewaystable} 

\clearpage
\newpage
\begin{table}[h!]
	\centering
	\setlength{\tabcolsep}{3pt} 
		\footnotesize
		\begin{tabular}{cccccccccccccccc} 
			\toprule
			\multicolumn{1}{c}{Instance} & 
			\multicolumn{1}{c}{$\mathcal{CTR}3$} & 
			\multicolumn{1}{C{1.5cm}}{Ewbank~{et al.}  2019 \cite{ewbank2019capacitated}} & 
			\multicolumn{1}{C{1.5cm}}{Ewbank~{et al.} 2016 \cite{ewbank2016unsupervised}} &
			\multicolumn{1}{C{1.5cm}}{Lysgaard~{et al.} \cite{lysgaard2004new}}&
			\multicolumn{1}{c}{Instance}& 
			\multicolumn{1}{c}{$\mathcal{CTR}3$}& 
			\multicolumn{1}{C{1.5cm}}{Ewbank~{et al.} 2019 \cite{ewbank2019capacitated}}& 
			\multicolumn{1}{C{1.5cm}}{Ewbank~{et al.} 2016 \cite{ewbank2016unsupervised}} & 
			\multicolumn{1}{C{1.5cm}}{Lysgaard~{et al.} \cite{lysgaard2004new}}
			\\
			\midrule		
			A-n32-k5 & 2.90 & 3.57& 16.19& 14.44& B-n41-k6 & 9.24& 4.72& 17.35& 31.31 \\ 
			A-n33-k5 & 2.62& 4.51 & 14.95& 18.77 &B-n43-k6 &4.90 & 5.25& 16.90& 77.84 \\
			A-n33-k6 & 3.16& 5.52& 16.55 & 27.56& B-n44-k7 & 3.39 & 5.52& 17.69 & 8.15\\
			A-n34-k5 & 2.25 & 4.70& 14.62& 19.07 & B-n50-k7 &7.70& 6.01 &22.12 &9.86\\
			A-n36-k5& 2.25& 3.92& 17.89& 34.13 & B-n52-k7 & 13.23 &5.87 &21.42& 24.25\\
			A-n37-k5 &3.65 & 4.31 & 14.17& 20.72 & B-n63-k10 & 33.29& 9.70 &36.16 &1783.43\\ 
			A-n37-k6 &3.65 & 5.88 & 16.86 & 316.46  &E-n22-k4 & 1.06 & 3.13& 7.49& 1.70\\
			A-n38-k5& 2.56 &  4.89 & 14.66 & 65.73& E-n23-k3 & 1.37 & 1.37& 2.41& 6.18\\
			A-n39-k5& 5.83 & 4.91& 14.83& 91.31& E-n30-k3& 1.77 & 2.49 & 7.13 & 14.8\\
			A-n39-k6 & 6.71 & 6.79& 17.11& 69.55 &E-n33-k4&  2.04& 3.35& 9.39& 15.82 \\ 
			A-n44-k6 & 13.06 & 6.71 & 18.7 & 243.22 & E-n51-k5& 14.04& 7.32& 16.55& 41.69\\ 
			A-n48-k7 & 17.66 & 9.11& 18.08 &148.40& E-n76-k7& 240.13 & 16& 31.26 & 8703.55\\
			A-n53-k7 & 25.15 & 11.6 & 21.51 & 131.11  &P-n20-k2& 1.04 & 2.12& 3.31 & 13.69\\ 
			A-n54-k7 & 26.57 & 8.47 & 20.67 & 1672.44 &P-n21-k2 & 1.09 & 2.05 & 3.30& 2.62\\
			A-n63-k10 & 51.60 & 13.94 & 34.67& 3976.79 &P-n22-k2 &1.19 & 1.87 & 3.81 &15.31 \\ 
			A-n64-k9 & 53.42 &16.53 &31.55 &3643.50 &P-n40-k5& 10.81 & 9.10 &13.83 &17.79 \\ 
			A-n69-k9 &112.11 & 12.41 &32.31 & 3268.00&P-n45-k5 & 14.24 & 7.68 & 13.96 &45.73\\ 
			A-n80-k10 & 196.61& 17.81& 42.73& 1973.50&P-n50-k7& 20.55& 13.69& 21.29 &290.17 \\ 
			B-n34-k5 & 2.42 & 3.75& 11.25 &30.78 &P-n55-k10 & 28.96& 10.79 &21.32& 2534.21  \\
			B-n35-k5 &  2.27 & 3.95& 12.64 & 3.97&P-n76-k4 & 90.56 & 6.05 &19.36 &195.53\\
			B-n38-k6 & 5.25 &4.98& 12.84 &30.27 &P-n76-k5 & 126.62 & 7.43 & 23.96 & 1622.78\\
			B-n39-k5 &  3.57& 3.97 & 12.43 & 8.37 & P-n101-k4 & 222.45& 6.83 &25.9 &155.92 \\
			\bottomrule
		\end{tabular}
	\caption{Comparison between $\mathcal{CTR}3$ results and baselines from the literature in terms of the runtime.}
	\label{table-runtime}
\end{table} 


After assessing the global performance of the proposed approach, in what follows, we will emphasize the appropriateness of the components used within the main proposed approach, specifically: centroids multi-start initialization, customized assignment metric, re-optimization process.

\subsection{Impact of centroids multi-start initialization} \label{multi-start-impact}
To assess the impact of the multi-start process for the centroids initialization within $\mathcal{CCBC}$, we first select a representative instance from each group $A, B, P$, namely : \textit{A-n53-k7, B-n68-k9, P-n76-k5}. Then, we run the proposed approach on each instance  while systematically varying the number of explored initial centroids combinations $\mathcal{N}_{it}$ in the $\mathcal{CCBC}$ step. Figure \ref{multi-start} reports the corresponding relative gap $\mathcal{GAP}_r$ as we incrementally explore a range of combinations. We observe that $\mathcal{GAP}_r$ decreases for all studied instances as long as $\mathcal{N}_{it}$ increases.
\begin{figure}[!htb]
	\centering
	\begin{minipage}{.5\textwidth}
		\centering
		 \includegraphics[width=1\linewidth]{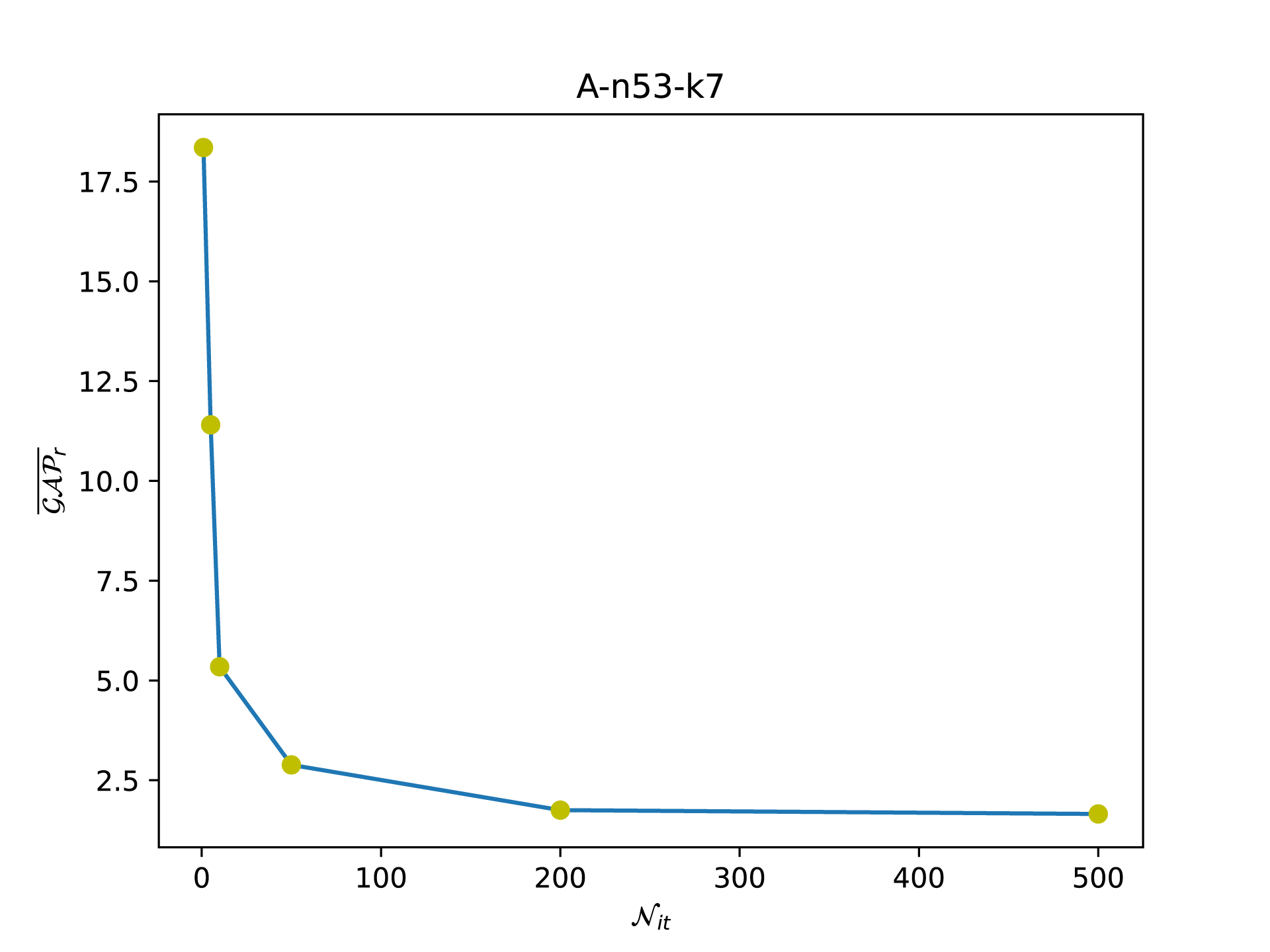}
	\end{minipage}%
	\begin{minipage}{0.5\textwidth}
		\centering
		 \includegraphics[width=1\linewidth]{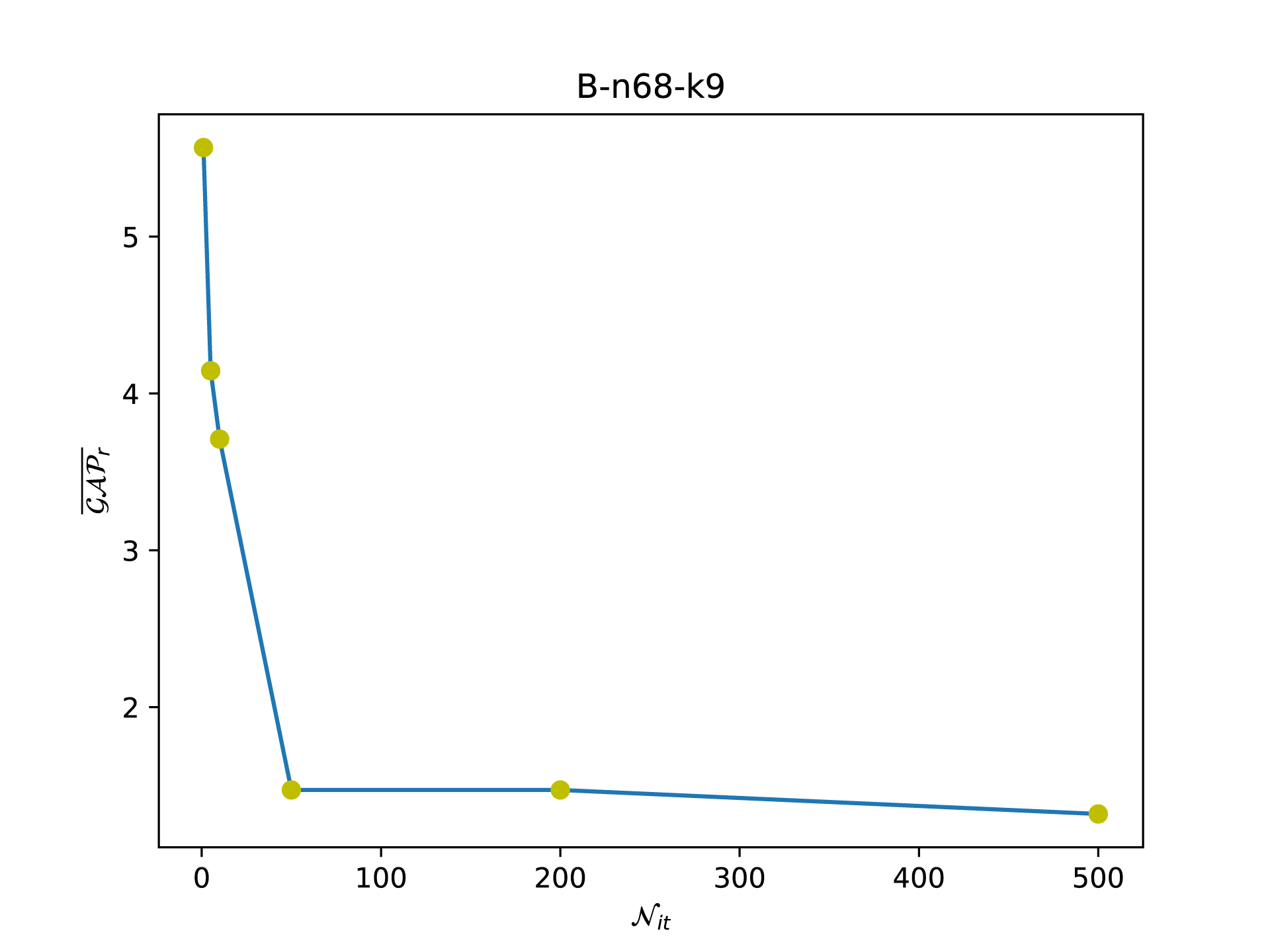}
	\end{minipage}
	\begin{minipage}{0.5\textwidth}
		\centering
		 \includegraphics[width=1\linewidth]{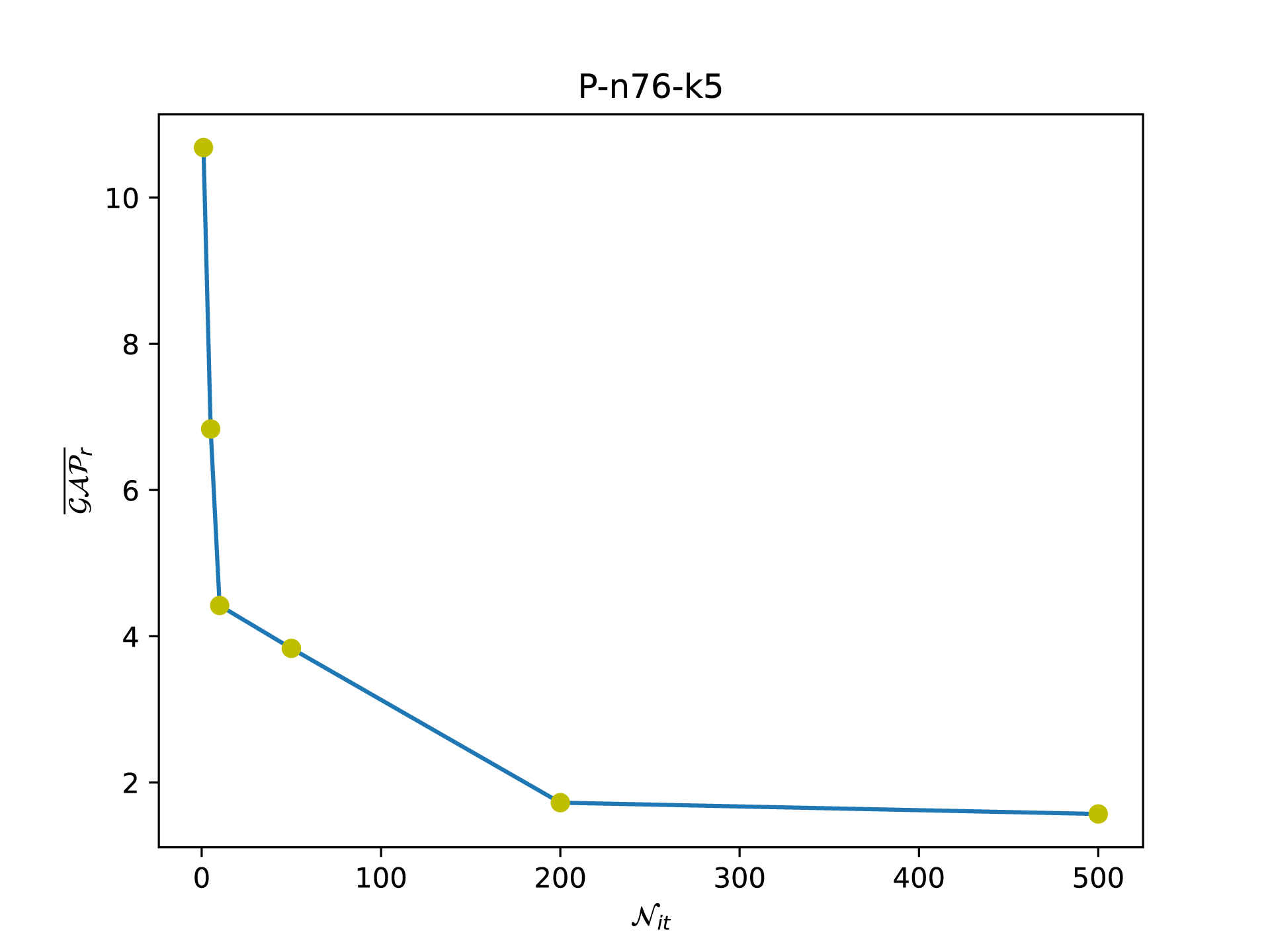}
	\end{minipage}
	\caption{Values of $\mathcal{GAP}_r$ varying the number of explored initial centroids combinations $\mathcal{N}_{it}$.}
	\label{multi-start}
\end{figure}

Furthermore, we provide a comparative study between the used initialization methodology and the well-known ones in machine learning, namely: \textit{Kmeans ++} and \textit{Naive sharding}. Specifically, we analyze the performance in terms of the relative gap $\overline{\mathcal{GAP}_r}$ for each initialization method per group. It should be noted that the comparison is carried out on instances from three groups $A, B, P$. Figure~\ref{initilzation} gives an overview on the $\overline{\mathcal{GAP}_r}$  results on average for each group using the three methodologies. We can notice that the multi-start initialization methodology consistently outperforms other ones in overage. 

In more detail, the average gap is above 7.5 \% for all groups when it comes to \textit{kmeans ++} or \textit{Naive sharding} initialization. Conversely, the same metric does not exceed 1.3 \% for any of the groups using a random multi-start initialization. This discrepancy with regard to the performance between the three approaches could be justified by the fact that \textit{kmeans ++} and \textit{Naive sharding} tend to select centroids in the same region, leading to convergence at a local optimum. In contrast, the multi-start approach ensures exploring diverse centroids combinations in the search space, therefore reaching a near-optimal solution for both the clustering step and then for $\mathcal{CVRP}$.  
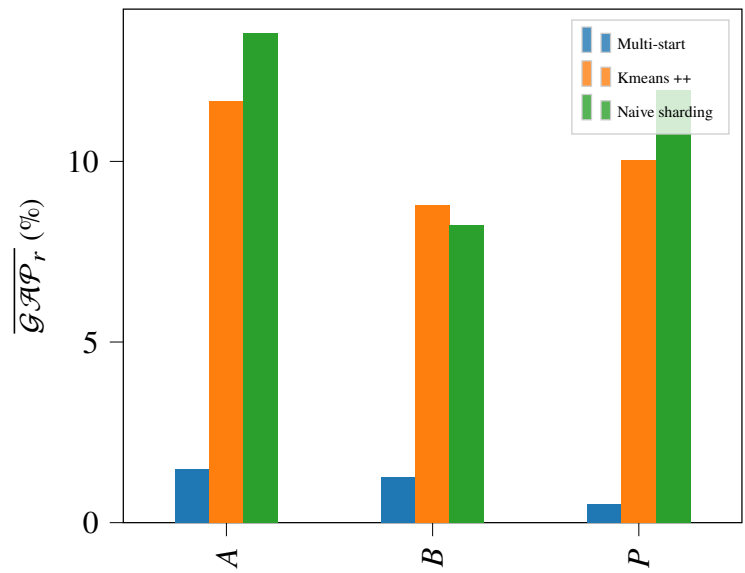
\begin{figure}[H]
	\begin{tikzpicture} [scale=1.2]
		
		\definecolor{darkgray176}{RGB}{176,176,176}
		\definecolor{darkorange25512714}{RGB}{255,127,14}
		\definecolor{forestgreen4416044}{RGB}{44,160,44}
		\definecolor{lightgray204}{RGB}{204,204,204}
		\definecolor{steelblue31119180}{RGB}{31,119,180}
		
		\begin{axis}[
			legend cell align={left},
			legend style={fill opacity=0.8, draw opacity=1, text opacity=1, draw=lightgray204},
			tick align=outside,
			tick pos=left,
			title={},
			x grid style={darkgray176},
			xlabel={},
			xmin=-0.5, xmax=2.5,
			xtick style={color=black},
			xtick={0,1,2},
			xticklabel style={rotate=90.0},
			xticklabels={$A$, $B$, $P$},
			y grid style={darkgray176},
			ylabel={\small{$\overline{\mathcal{GAP}_r}$ (\%)}},
			ymin=0, ymax=14.2147661090229,
			ytick style={color=black}
			]
			\draw[draw=none,fill=steelblue31119180] (axis cs:-0.25,0) rectangle (axis cs:-0.0833333333333333,1.46748797178297);
			\addlegendimage{ybar,ybar legend,draw=none,fill=steelblue31119180}
			\addlegendentry{\tiny{Multi-start}}
			
			\draw[draw=none,fill=steelblue31119180] (axis cs:0.75,0) rectangle (axis cs:0.916666666666667,1.2535949149035);
			\draw[draw=none,fill=steelblue31119180] (axis cs:1.75,0) rectangle (axis cs:1.91666666666667,0.501771638776899);
			\draw[draw=none,fill=darkorange25512714] (axis cs:-0.0833333333333333,0) rectangle (axis cs:0.0833333333333333,11.6730020471822);
			\addlegendimage{ybar,ybar legend,draw=none,fill=darkorange25512714}
			\addlegendentry{\tiny{Kmeans ++}}
			
			\draw[draw=none,fill=darkorange25512714] (axis cs:0.916666666666667,0) rectangle (axis cs:1.08333333333333,8.79954219453112);
			\draw[draw=none,fill=darkorange25512714] (axis cs:1.91666666666667,0) rectangle (axis cs:2.08333333333333,10.0452383208413);
			\draw[draw=none,fill=forestgreen4416044] (axis cs:0.0833333333333333,0) rectangle (axis cs:0.25,13.5378724847837);
			\addlegendimage{ybar,ybar legend,draw=none,fill=forestgreen4416044}
			\addlegendentry{\tiny{Naive sharding}}
			
			\draw[draw=none,fill=forestgreen4416044] (axis cs:1.08333333333333,0) rectangle (axis cs:1.25,8.2397087735331);
			\draw[draw=none,fill=forestgreen4416044] (axis cs:2.08333333333333,0) rectangle (axis cs:2.25,11.9702983018695);
		\end{axis}
		
	\end{tikzpicture}
	\caption{Comparison in terms of $\overline{\mathcal{GAP}_r}$ depending on the initialization methodology.}
	\label{initilzation}
\end{figure}

\subsection{Impact of assignment metric} \label{am-impact}
This section aims at highlighting the impact of the used metric to assign customers to clusters when compared to the classical assignment metric. This latter refers to the euclidean distance to assign nearest customers to clusters, i.e, \textit{K-means}. As illustrated, through Figure \ref{assignment}, choosing the assignment metric affects directly $\mathcal{CVRP}$ results. In detail, relying on a customized assignment metric provides better results than using the classical assignment metric. This performance concerns all groups as shown in this figure. The customized assignment metric is calculated using this formula $\mathcal{AM}(\mathscr{c}_i,S_k)=\frac{q_{i} }{d( \mathscr{c}_i,\mu_k^c)}$. It is designed in order to prioritize assigning nearest customers with high demands to the clusters. 

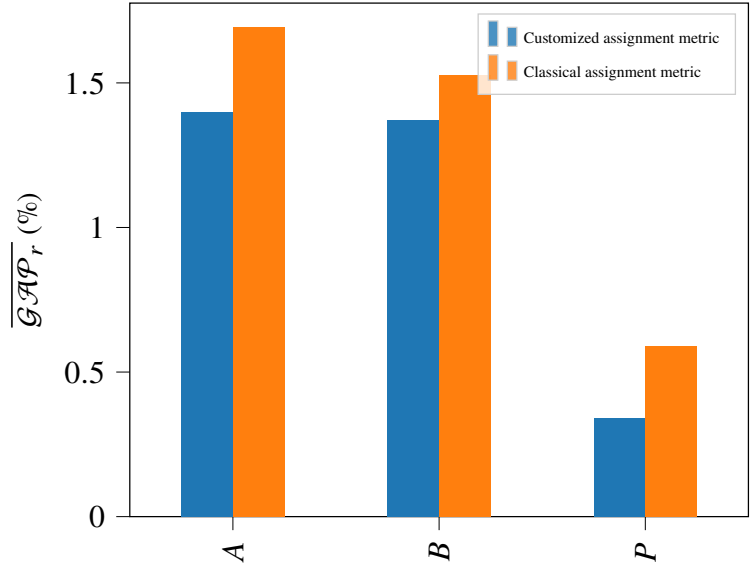
\begin{figure}[h!]
	\begin{tikzpicture} [scale=1.2]
		
		\definecolor{darkgray176}{RGB}{176,176,176}
		\definecolor{darkorange25512714}{RGB}{255,127,14}
		\definecolor{lightgray204}{RGB}{204,204,204}
		\definecolor{steelblue31119180}{RGB}{31,119,180}
		
		\begin{axis}[
			legend cell align={left},
			legend style={fill opacity=0.8, draw opacity=1, text opacity=1, draw=lightgray204},
			tick align=outside,
			tick pos=left,
			title={},
			x grid style={darkgray176},
			xlabel={},
			xmin=-0.5, xmax=2.5,
			xtick style={color=black},
			xtick={0,1,2},
			xticklabel style={rotate=90.0},
			xticklabels={$A$, $B$, $P$},
			y grid style={darkgray176},
			ylabel={\small{$\overline{\mathcal{GAP}_r}$ (\%)}},
			ymin=0, ymax=1.77600466616705,
			ytick style={color=black}
			]
			\draw[draw=none,fill=steelblue31119180] (axis cs:-0.25,0) rectangle (axis cs:0,1.3972569488454);
			\addlegendimage{ybar,ybar legend,draw=none,fill=steelblue31119180}
			\addlegendentry{\tiny{Customized assignment metric}}
			
			\draw[draw=none,fill=steelblue31119180] (axis cs:0.75,0) rectangle (axis cs:1,1.3700851759552);
			\draw[draw=none,fill=steelblue31119180] (axis cs:1.75,0) rectangle (axis cs:2,0.340789500481697);
			\draw[draw=none,fill=darkorange25512714] (axis cs:0,0) rectangle (axis cs:0.25,1.69143301539719);
			\addlegendimage{ybar,ybar legend,draw=none,fill=darkorange25512714}
			\addlegendentry{\tiny{Classical assignment metric}}
			
			\draw[draw=none,fill=darkorange25512714] (axis cs:1,0) rectangle (axis cs:1.25,1.52619123547446);
			\draw[draw=none,fill=darkorange25512714] (axis cs:2,0) rectangle (axis cs:2.25,0.589461758459689);
		\end{axis}
		
	\end{tikzpicture}
	\caption{Comparison in terms of $\overline{\mathcal{GAP}_r}$ depending on the used assignment metric.}
	\label{assignment}
\end{figure}

To gain deeper insights about the role of this customized assignment metric in enhancing the $\mathcal{CVRP }$ quality solutions, we propose to conduct an experiment to compare the vehicle fulfillment rate achieved $\textit{VF}$(\%) through using the classical assignment metric against the proposed customized metric. The results obtained from this experiment are highlighted in Figure \ref{2 figures} in the right. It presents the average of unfilled capacity inside the vehicles $\overline{\mathcal{GAP}_{r}^{VF}}$ for groups \textit{A, B, P} and for every assignment metrics. It is computed for a specific instance using the following formula: $\mathcal{GAP}_{r}^{VF}=\frac{K\times Q-\sum_{i=1}^N q_i}{K\times Q} \times 100$. As one can clearly notice that relying on a customized metric consistently guarantees a higher fulfillment rate for the vehicles in the $\mathcal{CVRP}$ context. These findings confirm that the enhanced vehicles capacity utilization can be achieved by adopting the previous metric when compared to the traditional one.

Vehicles fulfillment rate has a direct impact on the number of used vehicles when addressing $\mathcal{CVRP}$ instances. Thoroughly 
a higher fulfillment rate inherently leads to a reduced number of vehicles required to efficiently serve all customers. To illustrate this fact and how it variates depending on the used assignment metric, we perform another experiment to evaluate the number of vehicles depending on the assignment metric. Figure \ref{2 figures} in the left gives the number of instances $N_{K-K_{opt}}$ that are solved using a number of $K$ vehicles knowing that the number of vehicles used by the optimal solution is $K_{opt}$. In our case, we represent the number of instances with respect to $K-K_{opt} \in \{0,1,2,3,4\}$. This right Figure \ref{2 figures} shows that using the customized assignment metric guarantees a number of used vehicles near to the optimal number. 

\begin{figure}[H]
	\centering
	\begin{minipage}{.2\textwidth}
		\begin{tikzpicture}[scale=0.9]
			\definecolor{darkgray176}{RGB}{176,176,176}
			\definecolor{darkorange25512714}{RGB}{255,127,14}
			\definecolor{lightgray204}{RGB}{204,204,204}
			\definecolor{steelblue31119180}{RGB}{31,119,180}
			
			\begin{axis}[
				legend cell align={left},
				legend style={fill opacity=0.8, draw opacity=1, text opacity=1, draw=lightgray204},
				tick align=outside,
				tick pos=left,
				title={},
				x grid style={darkgray176},
				xlabel={},
				xmin=-0.5, xmax=2.5,
				xtick style={color=black},
				xtick={0,1,2},
				xticklabel style={rotate=90.0},
				xticklabels={$A$, $B$, $P$},
				y grid style={darkgray176},
				ylabel={$\overline{\mathcal{GAP}_{r}^{VF}}$(\%)},
				ymin=0, ymax=16.6606313131313,
				ytick style={color=black}
				]
				\draw[draw=none,fill=steelblue31119180] (axis cs:-0.25,0) rectangle (axis cs:0,9.28561140505585);
				\addlegendimage{ybar,ybar legend,draw=none,fill=steelblue31119180}
				\addlegendentry{\tiny{Customized assignment metric }}
				
				\draw[draw=none,fill=steelblue31119180] (axis cs:0.75,0) rectangle (axis cs:1,10.5404416839199);
				\draw[draw=none,fill=steelblue31119180] (axis cs:1.75,0) rectangle (axis cs:2,11.6917798294405);
				\draw[draw=none,fill=darkorange25512714] (axis cs:0,0) rectangle (axis cs:0.25,15.8672679172679);
				\addlegendimage{ybar,ybar legend,draw=none,fill=darkorange25512714}
				\addlegendentry{\tiny{Classical assignment metric}}
				
				\draw[draw=none,fill=darkorange25512714] (axis cs:1,0) rectangle (axis cs:1.25,14.5769464834682);
				\draw[draw=none,fill=darkorange25512714] (axis cs:2,0) rectangle (axis cs:2.25,14.8167525711595);
			\end{axis}
		\end{tikzpicture}
		\label{truck-load}
		
	\end{minipage}
	\hfill
	\begin{minipage}{.5\textwidth}
		\begin{tikzpicture}[scale=0.9]
			
			\definecolor{darkgray176}{RGB}{176,176,176}
			\definecolor{darkorange25512714}{RGB}{255,127,14}
			\definecolor{lightgray204}{RGB}{204,204,204}
			\definecolor{steelblue31119180}{RGB}{31,119,180}
			
			\begin{axis}[
				legend cell align={left},
				legend style={fill opacity=0.8, draw opacity=1, text opacity=1, draw=lightgray204},
				tick align=outside,
				tick pos=left,
				x grid style={darkgray176},
				xlabel={$K-K_{opt}$: gap with the optimal},
				xmin=-0.5, xmax=4.5,
				xtick style={color=black},
				xticklabel style={rotate=90.0},
				y grid style={darkgray176},
				ylabel={$N_{K-K_{opt}}$},
				ymin=0, ymax=55.65,
				ytick style={color=black}
				]
				\draw[draw=none,fill=steelblue31119180] (axis cs:-0.25,0) rectangle (axis cs:0,53);
				\addlegendimage{ybar,ybar legend,draw=none,fill=steelblue31119180}
				\addlegendentry{\tiny{Customized assignment metric}}
				
				\draw[draw=none,fill=steelblue31119180] (axis cs:0.75,0) rectangle (axis cs:1,14);
				\draw[draw=none,fill=steelblue31119180] (axis cs:1.75,0) rectangle (axis cs:2,0);
				\draw[draw=none,fill=steelblue31119180] (axis cs:2.75,0) rectangle (axis cs:3,3);
				\draw[draw=none,fill=steelblue31119180] (axis cs:3.75,0) rectangle (axis cs:4,0);
				\draw[draw=none,fill=darkorange25512714] (axis cs:0,0) rectangle (axis cs:0.25,29);
				\addlegendimage{ybar,ybar legend,draw=none,fill=darkorange25512714}
				\addlegendentry{\tiny{Classical assignement metric}}
				
				\draw[draw=none,fill=darkorange25512714] (axis cs:1,0) rectangle (axis cs:1.25,33);
				\draw[draw=none,fill=darkorange25512714] (axis cs:2,0) rectangle (axis cs:2.25,5);
				\draw[draw=none,fill=darkorange25512714] (axis cs:3,0) rectangle (axis cs:3.25,2);
				\draw[draw=none,fill=darkorange25512714] (axis cs:4,0) rectangle (axis cs:4.25,1);
			\end{axis}
		\end{tikzpicture}
	\end{minipage}
	\caption{Comparison in terms of $\overline{\mathcal{GAP}_r^{VF}}$ and $N_{K-K_{opt}}$ depending on the initialization methodology.}
	\label{2 figures}
\end{figure}
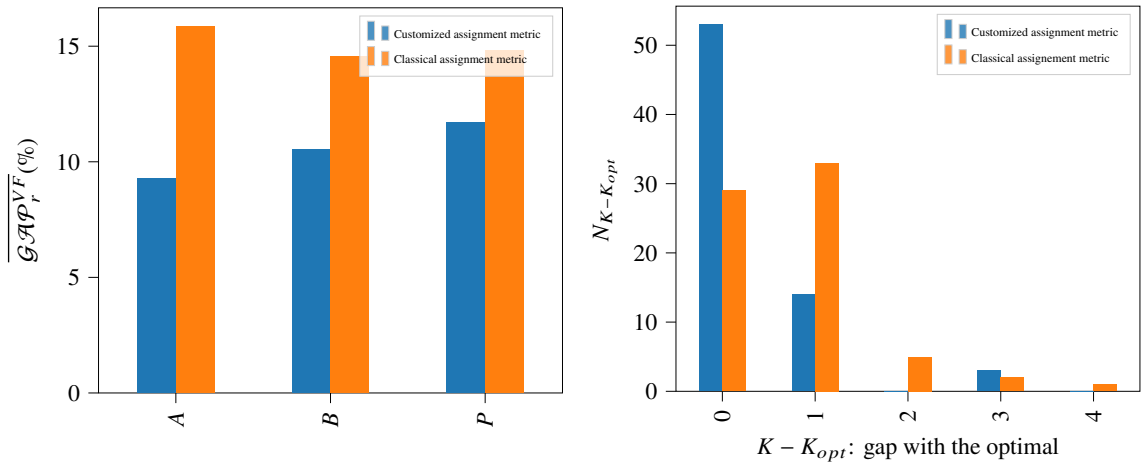 

\subsection{Impact of the third step of \texorpdfstring{$\mathcal{CTR}3$}{}} \label{reopt-impact}
To better understand the relevance of adding the ruin \& recreate step, we report in Figure \ref{reopt} the comparison results obtained using the proposed approach with the ruin \& recreate step(\textit{3-steps approach}) and without (\textit{2-steps approach}). One can clearly notice that using this last step through incorporating a ruining and recreating process significantly improves the solution quality in terms of the relative gap $\mathcal{GAP}_r$ with the optimal solution. This remarkable performance holds the same across various instances from all groups \textit{A, B, P} .

The clustering step is carried out through a heuristic method. Consequently, the existence of misclustered customers cannot be completely ruled out. It should be noted that even a solitary misplaced customer can substantially change the $\mathcal{CVRP}$ solution. To address this concern, the ruining and recreating mechanism was implemented to reassign these misclustered customers to the appropriate clusters, thereby enhancing the overall $\mathcal{CVRP}$ solution quality.

\begin{figure} [H]
	\begin{tikzpicture} [scale=1.2]
		
		\definecolor{darkgray176}{RGB}{176,176,176}
		\definecolor{darkorange25512714}{RGB}{255,127,14}
		\definecolor{lightgray204}{RGB}{204,204,204}
		\definecolor{steelblue31119180}{RGB}{31,119,180}
		
		\begin{axis}[
			legend cell align={left},
			legend style={fill opacity=0.8, draw opacity=1, text opacity=1, draw=lightgray204},
			tick align=outside,
			tick pos=left,
			title={},
			x grid style={darkgray176},
			xlabel={},
			xmin=-0.5, xmax=2.5,
			xtick style={color=black},
			xtick={0,1,2},
			xticklabel style={rotate=90.0},
			xticklabels={$A$, $B$, $P$},
			y grid style={darkgray176},
			ylabel={\small{$\overline{\mathcal{GAP}_{r}}$ (\%)}},
			ymin=0, ymax=3.15264706465791,
			ytick style={color=black}
			]
			\draw[draw=none,fill=steelblue31119180] (axis cs:-0.25,0) rectangle (axis cs:0,1.42143847781855);
			\addlegendimage{ybar,ybar legend,draw=none,fill=steelblue31119180}
			\addlegendentry{\tiny{3-steps approach}}
			
			\draw[draw=none,fill=steelblue31119180] (axis cs:0.75,0) rectangle (axis cs:1,1.35918857344465);
			\draw[draw=none,fill=steelblue31119180] (axis cs:1.75,0) rectangle (axis cs:2,0.344235859294362);
			\draw[draw=none,fill=darkorange25512714] (axis cs:0,0) rectangle (axis cs:0.25,3.00252101395992);
			\addlegendimage{ybar,ybar legend,draw=none,fill=darkorange25512714}
			\addlegendentry{\tiny{2-steps approach}}
			
			\draw[draw=none,fill=darkorange25512714] (axis cs:1,0) rectangle (axis cs:1.25,2.11442416060008);
			\draw[draw=none,fill=darkorange25512714] (axis cs:2,0) rectangle (axis cs:2.25,1.67840985867);
		\end{axis}
	\end{tikzpicture}
	\caption{Comparison in terms of $\overline{\mathcal{GAP}_r}$ using the proposed approach with and without the third step.}
	\label{reopt}
\end{figure}
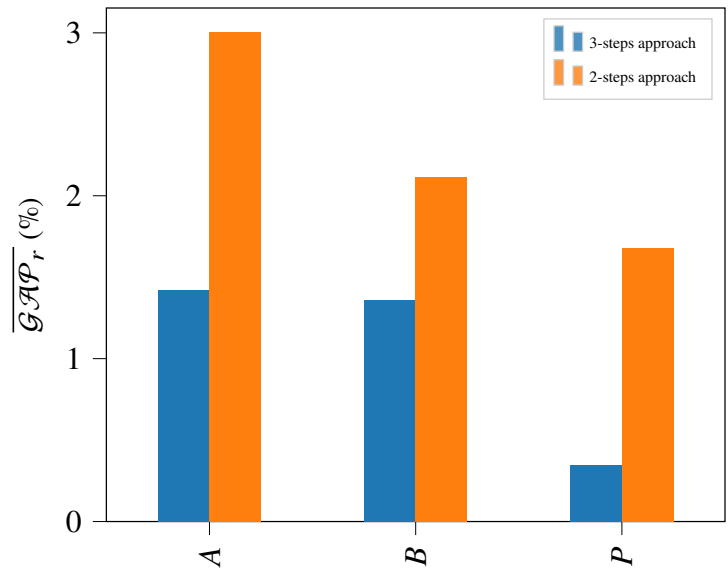


\section{Conclusion} \label{concl}
In conclusion, the goal of this study is to narrow the gap in understanding the connection between $\mathcal{CVRP}$ and $\mathcal{CCBC}$. Our findings demonstrate that optimal solutions to $\mathcal{CCBC}$ can offer valuable insights into solving $\mathcal{CVRP}$. The experimental results, corroborated by theoretical analysis, indicate a strong correlation between the optimal solutions of these two problems. This research paves the way for more efficient and practical approaches to solving the vehicle routing problems, leveraging the principles of centroid-based clustering. Future endeavors will involve delving deeper into the theoretical aspects of the clustering centroids regions leading to optimal or near-optimal solutions for
$\mathcal{CVRP}$. This includes a rigorous mathematical characterization of these regions, investigating properties such as their openness, convexity, connectedness. Additionally, we plan to explore the feasibility of reaching these regions while clustering through a reinforcement learning approach. This work will entail a dynamic evaluation of $\mathcal{CVRP}$ solutions in conjunction with the clustering process to effectively guide the clustering centroids selection.
 

\end{document}